\documentclass[review,1p]{elsarticle}
\usepackage{graphicx,amssymb,mathrsfs,amsmath,amsfonts,dsfont}
\usepackage{hyperref}
\usepackage{lineno}
\usepackage{subfigure}
\usepackage{algorithm}
\usepackage{algorithmic}
\usepackage{amsbsy}
\usepackage[mathscr]{euscript}
\usepackage{framed}

\DeclareMathOperator*{\argmin}{arg\,min}

\DeclareMathOperator*{\logdet}{log\,det}

\DeclareMathOperator*{\diag}{Diag}
\def\A{{\mathcal A}}
\def\B{{\mathcal B}}
\def\C{{\mathcal C}}
\def\D{{\mathcal D}}

\def\G{{\mathcal G}}
\def\H{{\mathcal H}}
\def\I{{\mathcal I}}
\def\J{{\mathcal J}}

\def\L{{\mathcal L}}
\def\M{{\mathcal M}}
\def\N{{\mathcal N}}

\def\P{{\mathcal P}}
\def\Q{{\mathcal Q}}

\def\S{{\mathcal S}}
\def\T{{\mathcal T}}

\def\W{{\mathcal W}}
\def\X{{\mathcal X}}
\def\Y{{\mathcal Y}}
\def\Z{{\mathcal Z}}
\makeatletter

\newcommand{\Rmnum}[1]{\expandafter\@slowromancap\romannumeral #1@}
\makeatother
\newtheorem{theorem}{Theorem}[section]
\newtheorem{lemma}{Lemma}[section]

\newtheorem{assumption}{Assumption}
\newproof{proof}{Proof}

\usepackage{tablefootnote}
\usepackage{url}

\usepackage{amsfonts}
\usepackage{mathrsfs}
\usepackage{bbm}

\usepackage{color}
\usepackage{graphicx}
\usepackage{epstopdf}
\usepackage{subfigure}
\usepackage{epsfig}
\usepackage{amssymb}
\usepackage{epsf}
\usepackage{subfigure}
\usepackage{multirow}
\usepackage{booktabs}
\usepackage{amsmath}
\usepackage{diagbox}

\journal{}

\begin{document}

\begin{frontmatter}

\title{A Majorized-Generalized Alternating Direction Method of Multipliers for Convex Composite Programming}



\author[a1]{Congying Qin}
\ead{qcymath@163.com}
\author[a2]{Yunhai Xiao}
\ead{yhxiao@henu.edu.cn}
\author[a3]{Peili Li}
\ead{plli@sfs.ecnu.edu.cn}

\address[a1]{School of Mathematics and Statistics,
Henan University, Kaifeng 475000, China}
\address[a2]{Henan Engineering Research Center for Artificial Intelligence Theory and Algorithms,
Henan University, Kaifeng 475000, China}
\address[a3]{School of Statistics, Key Laboratory of Advanced Theory and Application in Statistics and Data Science-MOE, East China Normal University, Shanghai 200062, P.R. China}

\begin{abstract}
The linearly constrained convex composite programming problems whose objective function contains two blocks with each block being the form of nonsmooth+smooth arises frequently in multiple fields of applications. If both of the smooth terms are quadratic, this problem can be solved efficiently by using the symmetric Gaussian-Seidel (sGS) technique based proximal alternating direction method of multipliers (ADMM). However, in the non-quadratic case, the sGS technique can not be used any more, which leads to the separable structure of nonsmooth+smooth had to be ignored. In this paper, we present a generalized ADMM and particularly use a majorization technique to make the corresponding subproblems more amenable to efficient computations.
Under some appropriate conditions, we prove its global convergence for the relaxation factor in $(0,2)$. We apply the algorithm to solve a kind of simulated convex composite optimization problems and a type of sparse inverse covariance matrix estimation problems which illustrates that the effectiveness of the algorithm are obvious.
\end{abstract}


\end{frontmatter}

\section{Introduction}\label{section1}
%
%
%
%
Let $\X:=\X_1\times\X_2\times\ldots\times\X_s$, $\Y:=\Y_1\times\Y_2\times\ldots\times\Y_t$ and $\Z$ be real finite dimensional Euclidean spaces each equipped with
an inner product $\langle\cdot,\cdot\rangle$ and its induced norm $\left\|\cdot\right\|$.
Let $x_i\in\X_i$ for each $i=1,\ldots,s$ and $y_j\in\Y_j$ for each $j=1,\ldots,t$, and write $x:=(x_1,x_2,\ldots,x_s)\in\X$ and $y:=(y_1,y_2,\ldots,y_t)\in\Y$.
Let $p(x)\equiv p_1(x_1)$ and $q(y)\equiv q_1(y_1)$ where $p_1:\X_1 \rightarrow(-\infty,+\infty]$ and $q_1:\Y_1 \rightarrow(-\infty,+\infty]$ are two simple closed proper
convex (not necessarily smooth) functions.
The primary motivation of this paper
is to develop an efficient method to solve the following convex composite optimization problem
\begin{equation}\label{pro}
\min_{x\in\X,y\in\Y}\bigg\{p(x)+f(x)+q(y)+g(y) \; | \; \A^*x+\B^*y=c \bigg\},
\end{equation}
where  $f:\X \rightarrow(-\infty,+\infty)$ and
$g :\Y \rightarrow (-\infty,+\infty)$ are two convex (not necessarily quadratic) functions with Lipschitz continous gradients
on $\X$ and $\Y$, respectively; $\A:\Z \rightarrow \X$ and $\B:\Z \rightarrow \Y$ are two linear maps with their adjoints $\A^*:\X \rightarrow \Z$ and $\B^*:\Y \rightarrow \Z$,
respectively; and $c\in \Z$ is a given data.
There exist many different variants of practical problems such as the robust principle component analysis of Wright et al. \cite{RPCA} satisfy the form (\ref{pro}), and generally, the Lagrangian dual of convex quadratic semidefinite programming and convex quadratic programming
falls within the form (\ref{pro}), see \cite{LIPHD}. Hence, studying problem (\ref{pro}) both theoretically and numerically is very important.

Let $\sigma\in (0,+\infty)$ be a penalty parameter.
For any $(x,y,z)\in \X\times \Y\times \Z$, the augmented Lagrangian function associated with problem (\ref{pro}) is given by
\begin{equation}\label{augf}
{\L}_{\sigma}(x,y;z)= p(x)+{f}(x)+q(y)+{g}(y)
+\langle z,\A^*x+\B^*y-c\rangle+\frac{\sigma}{2}\|\A^*x+\B^*y-c\|^2,
\end{equation}
where $z\in \Z$ be a multiplier associated with the constraint.
The alternating direction method of multipliers (ADMM) minimizes (\ref{augf}) alternatively with respect to $x$ and $y$ in
the following iterative form, see e.g., \cite{FAZ},
\begin{equation}\label{admm}
\left\{
\begin{array}{l}
x^{k+1}:=\argmin_{x\in \X}\: p(x)+f(x)+\frac{\sigma}{2}\|\A^*x+\B^*y^k-c+z^k/\sigma\|^2+\frac{1}{2}\|x-x^k\|_{\S}^2,
\\[2mm]
y^{k+1}:=\argmin_{y\in \Y }\: q(y)+g(y)+\frac{\sigma}{2}\|\A^*x^{k+1}+\B^*y-c+z^k/\sigma\|^2+\frac{1}{2}\|y-y^k\|_{\T}^2,
\\[2mm]
z^{k+1}:=z^{k}+\tau\sigma\big(\A^*x^{k+1}+\B^*y^{k+1}-c\big),
\end{array}
\right.
\end{equation}
where $\tau\in(0,(1+\sqrt{5})/2)$ is a step length; $\S$ and $\T$ are self-adjoint and positive
semidefinite linear operators. When $f$ and $g$ are convex quadratic functions, the presence
of the operator $\S$ and $\T$ may help to decompose the subproblem into some smaller problems, see \cite{LIPHD}.
Otherwise, due to the nonseparabilities of the non-quadratic functions $f$ and $g$, the
special structure of nonsmooth+smooth hidden in each block had to be ignored, therefore, it is a generally time consuming task to solve the $x$- and $y$-subproblems efficiently.

In optimization literature, one popular way to deal with a continuously differentially nonseparable convex function is the using of majorization \cite{ROC1}.
For example, Hong et al. \cite{HONG} considered linearly constrained convex optimization
problems whose objectives contain coupled functions and provided a majorized  ADMM with assuming that the step length $\tau$ is small enough.
The drawback of choosing a sufficient small step length is immediately remedied by Cui et al. \cite{CUI}.
Meanwhile, Li et al. \cite{LI} proposed a majorized ADMM for solving
(\ref{pro}) where $f$ and $g$ are not necessarily quadratic functions.
One distinctive feature of their majorized ADMM is that, if the symmetric Gaussian-Seidel (sGS) manner \cite{LIXD,SGSTH} is used, it generally decomposes the $x$- and $y$-subproblems of (\ref{admm}) into $2(s+t-1)$ smaller problems. Moreover, the quantity of the resulting smaller subproblems can be greatly reduced if the simple error tolerance criteria of Chen et al. \cite{CHEMMP} are used.

We should mention that all the methods reviewed above are based upon the classic ADMM originated from Glowinski \& Marroco \cite{GLOWINSKI75} and Gabay \& Mercier \cite{GABAY76} in the mid 1970s.
Later, Gabay \cite{GABAY83} showed that this type of ADMM with an unit step-length is equivalent to the Douglas-Rachford splitting (DRs) method \cite{DRS} to find the roots of the sum of two maximal monotone operators.
In early 1990s, Eckstein $\&$ Bertsekas \cite{ECK} showed that the DRs itself is an instance of the proximal point algorithm \cite{roc76} applied to a specially generated splitting operator.
Based on this crucial observation, Eckstein $\&$ Bertsekas \cite{ECK} introduced a generalized DRs with a over-relaxation factor which has been shown
to converge faster.
This in turn yielded a generalized-ADMM for convex programming which has been proved to be faster and also be
more widely applicable.
Recently, Xiao et al. \cite{XIAO} proposed a variant of generalized ADMM with semi-proximal terms, which can be used to solve various doubly non-negative semidefinite programming problems for moderate accuracy.

The successful applications of the generalized ADMM of Eckstein \& Bertsekas \cite{ECK} to solving multi-block problems inevitably inspire us to consider using it to solve the convex composite optimization problem (\ref{pro}). If proximal terms are added, the iterative scheme should follow the following form
\begin{equation}\label{gad}
\left\{
\begin{array}{l}
x^{k+1}:=\argmin\limits_{x\in \X}\: p(x)+f(x)+\frac{\sigma}{2}\|\A^*x+\B^*y^k-c+z^k/\sigma\|^2+\frac{1}{2}\|x-x^k\|_{\S}^2,
\\[2mm]
y^{k+1}:=\argmin\limits_{y\in \Y }\: q(y)+g(y)+\frac{\sigma}{2}\|\rho\A^*x^{k+1}-(1-\rho)\B^*y^k+\B^*y-c+z^k/\sigma\|^2+\frac{1}{2}\|y-y^k\|_{\T}^2,
\\[2mm]
z^{k+1}:=z^{k}+\sigma\big(\rho\A^*x^{k+1}-(1-\rho)\B^*y^k+\B^*y^{k+1}-c\big),
\end{array}
\right.
\end{equation}
where $\rho\in(0,2)$ is a relaxation factor. For $\rho=1$, the above iterative scheme is exactly \eqref{admm} with $\tau=1$.
However, also due to the nonseparabilities of $f$ and $g$, the favourable nonsmooth+smooth structure located in $x$- and $y$-subproblems can not be fully utilized.
To address this issue, as in \cite{LI,LINT}, we employ a majorization technique to iteratively
replace $f$ and $g$ by quadratic functions for the purpose of making both subproblems  more
amenable to efficient computations. In other words, when $f$ and $g$ are majorized by convex quadratic functions and the sGS technique of Li et al. \cite{LIXD,SGSTH} is used, then the $x$-subproblem ({\itshape resp. $y$-subproblem}) can be decomposed into pieces of smaller problems involving only the variable $x_i$ ({\itshape resp.} $y_j$) for each $i=1,\ldots,s$ ({\itshape resp.} $j=1,\ldots,t$).
To derive a general theoretical convergence result, in this paper, we particularly concentrate on the convergence analysis of the majorized and generalized ADMM for solving the convex composite optimization problem (\ref{pro}).
To implement the proposed algorithm, we also illustrate how to choose the positive semi-definite linear operators $\S$ and $\T$ for purpose of splitting each subproblem when the sGS technique is used.
Finally, using some simulated convex composite optimization problems and some sparse inverse covariance matrix estimation problems,
we do some numerical experiments  versus the majorized ADMM of Li et al. \cite{LI} and sGS-ADMM of Li $\&$ Xiao \cite{LIXIAO}.
The results demonstrate that the variant of generalized-ADMM with a special relaxation factor generally performs better.

The remaining parts of this paper are organized as follows.
In section \ref{section2}, we provide some useful results, which will be used for later theoretical analysis.
In section \ref{section3}, we present a majorized and generalized ADMM to solve the convex composite problem (\ref{pro}) followed by some important properties.
Moreover, the choosing of the proximal terms to split the subproblems is also included in this section.
Then, we focus on the convergence analysis of the proposed algorithm in section \ref{section4}.
In section \ref{section6}, we are devoted to implementation issue and illustrative examples to show the potential numerical efficiency of the proposed algorithm.
Finally, we conclude this paper with some remarks in section \ref{section7}.

\section{Preliminaries}\label{section2}

Note that $f$ and $g$ are smooth convex functions with Lipschitz continuous gradient, and their Hessian exist almost everywhere, then it is known that there exist two self-adjoint and positive semidefinite operators $\Sigma_f$ and $\Sigma_g$ such that
\begin{equation}\label{F1}
f(x)\geq f(x')+\langle x-x',\nabla f(x') \rangle+\frac{1}{2}\|x-x'\|_{\Sigma_f}^2, \ \ \forall \ x,x'\in \X,
\end{equation}
\begin{equation}\label{G1}
g(y)\geq g(y')+\langle y-y',\nabla g(y') \rangle+\frac{1}{2}\|y-y'\|_{\Sigma_g}^2, \ \ \forall \ y,y'\in \Y,
\end{equation}
and there also exist two self-adjoint and positive semidefinite operators $\widehat{\Sigma}_f $ and $\widehat{\Sigma}_g$ such that
\begin{equation}\label{F2}
f(x)\leq \hat{f}(x;x')=:f(x')+\langle x-x',\nabla f(x') \rangle+
\frac{1}{2}\|x-x'\|_{\widehat{\Sigma}_f}^2, \ \ \forall \ x,x'\in \X,
\end{equation}
\begin{equation}\label{G2}
g(y)\leq \hat{g}(y;y')=: g(y')+\langle y-y',\nabla g(y') \rangle+
\frac{1}{2}\|y-y'\|_{\widehat{\Sigma}_g}^2, \ \ \forall \ y,y'\in \Y,
\end{equation}
where $\hat{f}$ and $\hat{g}$ are called the majorized convex functions of $f$ and $g$, respectively.
The Clarke's generalized Hessian of $f$ at $x\in \X$ is defined as
$$
\partial^2 f(x):= \text{conv}\Big\{ \lim_{x^k \rightarrow x} \nabla^2 f(x^k): \nabla^2 f(x^k) \ \text{exists} \Big\},
$$
where conv denotes the convex hull. It is known from \cite[Proposition 2.6.2]{CLA} that $\partial^2 f(x)$ is a nonempty convex compact set.
Let $W\in \partial^2 f(x)$ be a self-adjoint and positive semidefinite linear operator,
the second order Mean-Value Theorem indicates that, for
$x,x'\in \X$, there exist $\omega\in[x',x]$ and $W\in \partial^2 f(\omega)$ such that
$$
f(x)=f(x')+\langle x-x',\nabla f(x') \rangle+\frac{1}{2}\|x-x'\|_{W}^2.
$$
Since $\nabla f$ is assumed to be globally Lipschitz continuous, then each given $ W\in \partial^2 f(x)$ satisfies
\begin{equation}\label{GG}
\Sigma_f \leq  W \leq \widehat{\Sigma}_f ,
\end{equation}
where $\Sigma_f$ and $\widehat{\Sigma}_f$ are independent of $x$.

Denote $\varphi(x):=p(x)+f(x)$ and $\phi(y):=q(y)+g(y)$, then (\ref{pro}) is rewritten as the following convex minimization problem
\begin{equation}\label{mp}
\begin{array}{ll}
\min\limits_{x\in\X,y\in\Y } & \varphi(x)+\phi(y) \\
\mbox{s.t.} & \A^* x+\B^* y=c,
\end{array}
\end{equation}
where $\varphi:\mathcal{X}\rightarrow (-\infty,+\infty]$ and $\phi:\mathcal{Y}\rightarrow (-\infty,+\infty]$ are closed proper convex. The dual of problem (\ref{mp}) is equivalent to
\begin{equation}\label{dmp}
\min_{z\in\Z} \Big\{\langle c,z \rangle+\varphi^{*}(-\A z)+\phi^{*}(-\B z)\Big\},
\end{equation}
where $\varphi^*(\cdot)$ and $\phi^*(\cdot)$ are Fenchel conjugate of $\varphi(\cdot)$ and $\phi(\cdot)$, respectively. Assuming that the Karush-Kuhn-Tucker (KKT) system associated with (\ref{mp}) and (\ref{dmp}) is not empty, then solving (\ref{dmp}) is equivalent to finding zeros for the following inclusion problem
\begin{equation}\label{mpin}
0\in (\T_1+\T_2)(z),
\end{equation}
where
$$
\T_1(z)=c-\A^*\partial \varphi^*(-\A z),\quad\text{and}\quad\T_2(z)=-\B^*\partial \phi^*(-\B z).
$$
Denote $\J_{c\T}:=(\I+c\T)^{-1}$, i.e., a resolvent operator, then the DRs  \cite{DRS} for solving (\ref{mpin}) is with the form
\begin{equation}\label{drs}
z^{k+1}=\J_{c\T_2}(w^{k+1}) \quad \text{with}\quad w^{k+1}=[\J_{c\T_1}\circ(2\J_{c\T_2}-\I)+(\I-\J_{c\T_2})](w^k) \quad\big(=:\G_{c\T_1\T_2}(\omega_k)\big),
\end{equation}
where $\circ$ denotes functional composition.
Denote $\S_{c\T_1\T_2}:=\G_{c\T_1\T_2}^{-1}-\I$, where $\I$ is an identity operator, then the DRs (\ref{drs}) can be rewritten as
$
w^{k+1}=(\I+\S_{c\T_1\T_2})^{-1}(w^k),
$
which is the proximal point algorithm of Rockfellar \cite{roc76}.
Instead of (\ref{drs}), Eckstein $\&$ Bertsekas \cite{ECK} considered a generalized version of DRs with form
\begin{equation}\label{gdrs}
w^{k+1}=(1-\rho)w^k+\rho\G_{c\T_1\T_2}(w^k),
\end{equation}
where $\rho\in(0,2)$ is called a relaxation factor. It is from \cite[Theorem 8]{roc76} or \cite[Theorem 4.1]{WUC} that, the generalized DRs (\ref{gdrs}) along with the relation $z^{k+1}=\J_{c\T_2}(w^{k+1})$ is exactly the iterative scheme (\ref{gad}) by noting that $\varphi(x)\equiv p(x)+f(x)$ and $\phi(y)\equiv q(y)+g(y)$.

Throughout this paper, the following condition is assumed to be true.
\begin{assumption}\label{assum}
Assume that there exists at least one $(x,y)$ such that $(x,y)\in ri(dom(p)\times dom(q))\cap P $ with $P:=\{(x,y)\in\X \times \Y \; |\; \A^*x+\B^*y=c$\}.
\end{assumption}

Under this assumption, it follows from \cite[Corollaries 28.2.2 and 28.3.1]{ROC1} that $(\bar{x},\bar{y})$ is an optimal solution to problem (\ref{pro}) if and only if there exist a $\bar{z}\in \Z $ such that $(\bar{x},\bar{y},\bar{z})\in\W:=\X\times\Y\times\Z$ satisfies the following Karush-Kuhn-Tucker (KKT) system
\begin{equation}\label{KKT}
0\in \partial p(\bar{x})+\nabla f(\bar{x})+\A\bar{z}, \quad 0\in \partial q(\bar{y})+\nabla g(\bar{y})+\B\bar{z}
\quad\mbox{and}\quad \A^*\bar{x}+\B^*\bar{y}=c,
\end{equation}
where $\partial p(x)$ and $\partial q(y) $ are the subdifferential mappings of $p$ and $q$, respectively.
Recall that $p$ and $q$ are convex functions, then system (\ref{KKT}) is equivalent to the following variational inequality
\begin{equation}\label{VA}
(p(x)+q(y))-(p(\bar{x})+q(\bar{y}))+
\langle x-\bar{x},\nabla f(\bar{x})+\A\bar{z} \rangle + \langle y-\bar{y},\nabla g(\bar{y})+\B\bar{z} \rangle
\geq 0, \ \ \forall (x,y)\in\X\times\Y.
\end{equation}
We denote by $\W^*$ the solution set of (\ref{VA}) which is nonempty under Assumption \ref{assum}.

\section{A majorized and generalized ADMM}\label{section3}

We have known that the ADMM minimizes the augmented Lagrangian function (\ref{augf}) alternatively with respect to $x$ and $y$, and takes the form (\ref{admm}) if the classic version of \cite{FAZ} is used and the form  (\ref{gad}) if the generized variant of \cite{ECK} is used.
However, it is generally time consuming if $f$ and $g$ are not quadratic functions because the separated structure located in each subproblem can not be fully utilized. To address this issue, Li et al. \cite{LI} used a majorized augmented Lagrangian function to replace (\ref{augf}) and proved that the corresponding algorithm (\ref{admm}) converges globally under some appropriate conditions.
As a re-product,  we mainly concentrate on the generalized ADMM iterative scheme (\ref{gad}) and show that if the majorization technique used, the attractive properties as \cite{LI} are also possessed.

\subsection{Algorithm's construction}

For the sake of simplicity, we denote
$$
\P:=\small{\widehat{\Sigma}_f}+\S+\sigma\A\A^* \quad \mbox{and} \quad  \Q:=\small{\widehat{\Sigma}_g}+\T+\sigma\B\B^*.
$$
On the one hand, at current iteration, we use the majorized function $\hat f(x;x^k)$ to replace the non-quadratic function $f$ in the $x$-subproblem of (\ref{gad}), that is
\begin{align*}
x^{k+1}:=&\argmin\limits_{x\in \X} \ \Big\{p(x)+\hat f(x;x^k)+\frac{\sigma}{2}\|\A^*x+\B^*y^k-c+z^k/\sigma\|^2+\frac{1}{2}\|x-x^k\|_{\S}^2\Big\}\\
        =&\argmin\limits_{x\in\X} \ \Big\{p(x)+\frac{1}{2}\langle x,\P x\rangle
+\big\langle \nabla f(x^k)+\sigma\A(\A^*x^k+\B^*y^k-c+\sigma^{-1}z^k)-\P x^k ,x \big\rangle\Big\}.
\end{align*}
On the other hand, we use $\hat g(y;y^k)$ to replace $g$ in the $y$-subproblem of (\ref{gad}), that is
\begin{align*}
y^{k+1}:=&\argmin\limits_{y\in \Y } \  \Big\{q(y)+\hat g(y;y^k)+\frac{\sigma}{2}\|\rho\A^*x^{k+1}-(1-\rho)\B^*y^k+\B^*y-\rho c+z^k/\sigma\|^2+\frac{1}{2}\|y-y^k\|_{\T}^2\Big\}\\
        =&\argmin\limits_{y\in\Y} \ \Big\{q(y)+\frac{1}{2}\langle y,\Q y \rangle
+\big\langle \nabla g(y^k)+\sigma\rho\B(\A^*x^k+\B^*y^k-c+(\sigma\rho)^{-1}z^k)\\
&\qquad \qquad \quad+\sigma\rho\B\A^*(x^{k+1}-x^k)-\Q y^k ,y \big\rangle\Big\}.
\end{align*}
Comparing with (\ref{gad}), it is clear to see that both subproblems have the character of nonsmooth+quadratic.
We will show later that if a sGS technique is used, this favorable structure can guarantee to decompose into some smaller problems.
In light of the above analyses, the full steps of the majorized and generalized ADMM for solving the convex composite optimization problem
(\ref{pro}) (abbr. MGADMM) are summarized as follows:

\begin{algorithm}
\noindent
{\bf Algorithm: MGADMM}
\vskip 1.0mm \hrule \vskip 1mm
\noindent
\textbf{Step 0.} Let $\rho\in(0,2)$ and $\sigma>0$. Choose $\S:\X\rightarrow\X$ and $\T:\Y\rightarrow\Y$ such that $\P\succeq 0$ and $\Q\succeq 0$, i.e., self-adjoint and positive definite. Choose an initial point $(x^0,y^0,z^0)\in dom(p) \times dom(q) \times \Z $. Let $k:=0$.\\
\textbf{Step 1.} Terminate some stopping criterions are satisfied. Otherwise, compute
\begin{equation}\label{mgadmm}
\left\{
\begin{array}{l}
x^{k+1}:=\argmin\limits_{x\in\X}  \Big\{p(x)+\frac{1}{2}\langle x,\P x \rangle
+\big\langle \nabla f(x^k)+\sigma\A(\A^*x^k+\B^*y^k-c+\sigma^{-1}z^k)-\P x^k ,x \big\rangle\Big\}\\
y^{k+1}:=\argmin\limits_{y\in\Y}  \Big\{q(y)+\frac{1}{2}\langle y,\Q y \rangle
+\big\langle \nabla g(y^k)+\sigma\rho\B(\A^*x^k+\B^*y^k-c+(\sigma\rho)^{-1}z^k)\\[2mm]
\qquad \qquad \quad+\sigma\rho\B\A^*(x^{k+1}-x^k)-\Q y^k ,y \big\rangle\Big\}\\[2mm]
z^{k+1}:=z^k+\sigma\Big(\rho\A^*x^{k+1}-(1-\rho)\B^*y^k+\B^*y^{k+1}-\rho c\Big).
\end{array}
\right.
\end{equation}
\textbf{Step 2.} Let $k:=k+1$ and go to Step 1.
\end{algorithm}

We discuss how to choose the proximal terms $\S$ and $\T$ to decompose each subproblem into some smaller subproblems.
We only take the $x$-subproblem as an example, because the $y$-subproblem can be solved in a similar way.
Noting that $\P\equiv\widehat{\Sigma}_f+\S+\sigma\A\A^*$,
then the $x$-subproblem in (\ref{mgadmm}) reduces to
\begin{align*}
x^{k+1}&=\argmin\limits_{x\in\X} \ \Big\{p(x_1)+\frac{1}{2}\big\langle x,(\widehat{\Sigma}_f+\sigma\A\A^*) x\big\rangle\\
&+\big\langle \nabla f(x^k)+\sigma\A(\B^*y^k-c+\sigma^{-1}z^k)-\widehat{\Sigma}_f x^k, x \big\rangle+\frac12\|x-x^k\|^2_{\S}\Big\}.
\end{align*}
For convenience, we denote
$$
\H:=\widehat{\Sigma}_f+\sigma\A\A^*,\quad\text{and}\quad r:=\nabla f(x^k)+\sigma\A(\B^*y^k-c+\sigma^{-1}z^k)-\widehat{\Sigma}_f x^k,
$$
and define
$$
h(x)=\frac12\langle x,\H x\rangle +\langle r,x\rangle.
$$
If $\H$ is further assumed to be positive definite, then $h$ is a strictly convex quadratic function. For $\H$, we split it as $\H=\M+\D+\M^*$ with
$$
\M=\left(
\begin{array}{cccc}
0    &    \H_{12}    &     \cdots    &    \H_{1s}     \\
     &    \ddots     &     \cdots    &    \vdots      \\
     &               &     0         &    \H_{(s-1)s} \\
     &               &               &      0
\end{array}
\right), \ \text{and} \
\D=\diag(\H_{11},\cdots,\H_{ss}).
$$
It is from \cite[Theorem 1]{SGSTH} that, if we specifically choose $\S=\M\D^{-1}\M^*$ then the optimal solution $x^{k+1}$ to $x$-subproblem in (\ref{mgadmm})
can be obtained by two stages, that is, firstly computing $\tilde x^k_i\in\X_i$ in an order of $i=s,\ldots,2$ by
$$
\tilde x^k_i:=\argmin_{x_i\in\X_i} h(x^k_{<i},x_i,\tilde x^k_{>i}),
$$
and then computing $x^{k+1}_i\in\X_i$ in an order of $i=1,\ldots,s$ by
\begin{align*}
x^{k+1}_1:=&\argmin_{x_i\in\X_i} \ p(x_1)+ h(x_1,\tilde x^k_{>1})\\
x^{k+1}_i:=&\argmin_{x_i\in\X_i} \ h(x^{k+1}_{<i},x_i,\tilde x^k_{>i}).
\end{align*}
The approach used above shows that if a general function $f(x)$ is replaced by its majorized function $\hat f(x;x^k)$, then the larger problem
can be decomposed into $2s-1$ smaller parts with respect to each $x_i\in\X_i$ and then solves it correspondingly via its favorable structure.
This is the well-known sGS technique of Li et al. \cite{LIXD,SGSTH}, which
has been widely and successfully used to solve many multi-block conic programming problems over the recent years, such as \cite{CHEMMP,XIAO}.

\subsection{Some useful lemmas}
We now report some key inequalities used to prove the convergence results in the next section.
For notational convenience,  we denote
\begin{equation}\label{MN}
\M_f:=\frac{1}{2}\Sigma_f+\S+\frac{1}{2}\sigma(1-\lambda)(2-\rho)\A\A^*, \quad \text{and} \quad
\N_g:=\frac{1}{2}\Sigma_g+\T+\sigma\frac{\lambda(2-\rho)^2}{\rho}\B\B^*,
\end{equation}
where $\lambda\in(0,1]$. For any $(x,y,z)\in \W$,  we also denote
\begin{equation}\label{f}
\phi_k(x,y,z):=(\sigma\rho)^{-1}\|z^k-z+\sigma(\rho-1)\B^*(y^k-y)\|^2+\|x^k-x\|^2_{\widehat{\Sigma}_f+\S}
+\|y^k-y\|^2_{\widehat{\Sigma}_g+\T+\sigma(2-\rho)\B\B^*},
\end{equation}
and
\begin{align}
\theta_k:=&\|x^k-x^{k-1}\|^2_{\frac{1}{2}\Sigma_f+\S}+\|y^k-y^{k-1}\|^2_{\frac{1}{2}\Sigma_g+\T},\label{t}\\[2mm]
\vartheta_k:=&\|x^k-x^{k-1}\|^2_{\M_f}+\|y^k-y^{k-1}\|^2_{\N_g},\label{v}\\[2mm]
\zeta_k:=&\|y^k-y^{k-1}\|^2_{\widehat{\Sigma}_g+\T}.\label{z}
\end{align}
Besides, the following elementary identities are used frequently in the subsequent theoretical analysis.
\begin{lemma}\label{lem1}
For any vectors $u, v$ in the same finite dimensional real Euclidean space,  it holds that
\begin{equation}\label{iec}
\frac{1}{2}\|u\|_{\G}^2 +\frac{1}{2}\|v\|_{\G}^2 \geq  \frac{1}{4}\|u-v\|_{\G}^2,
\end{equation}
and
\begin{equation}\label{id2}
2\langle u,\G v \rangle=\|u\|_{\G}^2+\|v\|_{\G}^2-\|u-v\|_{\G}^2
=\|u+v\|_{\G}^2-\|u\|_{\G}^2-\|v\|_{\G}^2,
\end{equation}
where $\mathcal{G}$ is an arbitrary self-adjoint positive semidefinite linear operator.
\end{lemma}

Recall that $(\bar{x},\bar{y},\bar{z})\in \W^*$ be an arbitrary solution to the KKT system (\ref{KKT}).
For convenience, we denote
$x_{e}:=x-\bar{x}$, $y_{e}:=y-\bar{y}$ and $z_{e}:=z-\bar{z}$ for any $(x,y,z)\in \W$.
The following three lemmas play key rules in establishing the convergence theorem of algorithm MGADMM.
\begin{lemma}\label{lem3}
Let $\{(x^k, y^k, z^k )\}$ be the sequence generated by algorithm MGADMM and $(\bar{x},\bar{y},\bar{z})\in \W^*$. For any $\rho\in(0,2)$ and $k\geq0$, it holds that
\begin{equation}\label{l3}
\begin{aligned}
&\Big\langle\A^*x^{k+1}_{e}+\B^*y^{k+1}_{e}, z^{k+1}_{e}+\sigma(\rho-1)\B^*y^{k+1}_{e}\Big\rangle\\[2mm]
=&\frac{1}{2\sigma\rho}\Big[\|z^{k+1}_{e}+\sigma(\rho-1)\B^*y^{k+1}_{e}\|^2
-\|z^k_{e}+\sigma(\rho-1)\B^*y^k_{e}\|^2\Big]
+\frac{\sigma\rho}{2}\|\A^*x^{k+1}_{e}+B^*y^{k+1}_{e}\|^2.\\[2mm]
\end{aligned}
\end{equation}
\end{lemma}
\begin{proof}
From the third implementation in (\ref{mgadmm}), we get
\begin{equation}\label{zz}
\begin{aligned}
z^{k+1}&=z^{k}+\sigma\big(\rho\A^*x^{k+1}-(1-\rho)\B^*y^{k}+\B^*y^{k+1}-\rho c\big)\\[2mm]
&=z^{k}+\sigma\rho\big(\A^*x^{k+1}+\B^*y^{k+1}-c\big)-\sigma(\rho-1)\B^*(y^{k+1}-y^k),
\end{aligned}
\end{equation}
which means
\begin{equation}\label{zz22}
\A^*x^{k+1}+\B^*y^{k+1}-c=\frac{1}{\sigma\rho}\Big(z^{k+1}-z^{k}+\sigma(\rho-1)\B^*(y^{k+1}-y^k)\Big).
\end{equation}
The equality (\ref{zz}) is equivalent to
$$
[z^{k+1}+\sigma(\rho-1)\B^*y^{k+1}]-[z^k+\sigma(\rho-1)\B^*y^k]
=\sigma\rho(\A^*x^{k+1}+\B^*y^{k+1}-c),
$$
and hence
\begin{equation}\label{zz1}
[z^{k+1}_{e}+\sigma(\rho-1)\B^*y^{k+1}_{e}]-[z^k_{e}+\sigma(\rho-1)\B^*y^k_{e}]
=\sigma\rho(\A^*x^{k+1}_{e}+\B^*y^{k+1}_{e}).
\end{equation}
Using the first identity in (\ref{id2}) and using (\ref{zz1}), we get
\begin{equation}\label{l31}
\begin{aligned}
&2\sigma\rho\langle \A^*x^{k+1}_{e}+\B^*y^{k+1}_{e},z^{k+1}_{e}+\sigma(\rho-1)\B^*y^{k+1}_{e}\rangle\\[2mm]
=&\sigma^2\rho^2\|\A^*x^{k+1}_{e}+\B^*y^{k+1}_{e}\|^2+\|z^{k+1}_{e}+\sigma(\rho-1)\B^*y^{k+1}_{e}\|^2
-\|z^k_{e}+\sigma(\rho-1)\B^*y^k_{e}\|^2,
\end{aligned}
\end{equation}
which is exactly the assertion (\ref{l3}).
\end{proof}

\begin{lemma}\label{lem4}
Let $\{(x^k, y^k, z^k )\}$ be the sequence generated by algorithm MGADMM and $(\bar{x},\bar{y},\bar{z})\in \W^*$. For any $\rho\in(0,2)$ and $k\geq0$, it holds that
\begin{equation}\label{l4}
\begin{aligned}
&\langle\A^*x^{k+1}_{e},\sigma(\rho-1)(A^*x^{k+1}+\B^*y^k-c)+\sigma\B^*(y^{k+1}-y^k)\rangle\\[2mm]
=&\frac{\sigma(\rho-1)}{2}\|\A^*x^{k+1}_{e}\|^2-\frac{\sigma(\rho-2)}{2}\|\B^*y^k_{e}\|^2
-\frac{\sigma}{2}\|\B^*y^{k+1}_{e}\|^2\\[2mm]
&\quad+\frac{\sigma(\rho-2)}{2}\|\A^*x^{k+1}_{e}+\B^*y^k_{e}\|^2+\frac{\sigma}{2}\|\A^*x^{k+1}_{e}+\B^*y^{k+1}_{e}\|^2.
\end{aligned}
\end{equation}
\end{lemma}
\begin{proof}
Noticing the fact $\A^*\bar{x}+\B^*\bar{y}=c$, we obtain
\begin{equation}\label{l41}
\begin{aligned}
&\langle\A^*x^{k+1}_{e},\sigma(\rho-1)(A^*x^{k+1}+\B^*y^k-c)+\sigma\B^*(y^{k+1}-y^k)\rangle\\[2mm]
=&\langle\A^*x^{k+1}_{e},\sigma(\rho-1)A^*x^{k+1}_{e}+\sigma(\rho-2)B^*y^k_{e}+\sigma\B^*y^{k+1}_{e}\rangle\\[2mm]
=&\sigma(\rho-1)\|\A^*x^{k+1}_{e}\|^2+\sigma(\rho-2)\langle\A^*x^{k+1}_{e},B^*y^k_{e}\rangle
+\sigma\langle\A^*x^{k+1}_{e},\B^*y^{k+1}_{e}\rangle.
\end{aligned}
\end{equation}
Using the the second identity in (\ref{id2}), we have
\begin{equation}\label{l42}
\begin{aligned}
&\sigma(\rho-2)\langle\A^*x^{k+1}_{e},B^*y^k_{e}\rangle
=\frac{\sigma(\rho-2)}{2}\big(\|\A^*x^{k+1}_{e}+\B^*y^k_{e}\|^2
-\|\A^*x^{k+1}_{e}\|^2-\|\B^*y^k_{e}\|^2\big),
\end{aligned}
\end{equation}
and
\begin{equation}\label{l43}
\begin{aligned}
&\sigma\langle\A^*x^{k+1}_{e},\B^*y^{k+1}_{e}\rangle
=\frac{\sigma}{2}\big(\|\A^*x^{k+1}_{e}+\B^*y^{k+1}_{e}\|^2-\|\A^*x^{k+1}_{e}\|^2-\|\B^*y^k_{e}\|^2\big).\\[2mm]
\end{aligned}
\end{equation}
Then, substituting (\ref{l42}) and (\ref{l43}) into (\ref{l41}), we can get (\ref{l4}).
\end{proof}

\begin{lemma}\label{lem5}
Let $\{(x^k, y^k, z^k )\}$ be the sequence generated by algorithm MGADMM and $(\bar{x},\bar{y},\bar{z})\in \W^*$. For any $\rho\in(0,2)$ and $k\geq0$, it holds that
\begin{equation}\label{in}
\phi_k(\bar{x},\bar{y},\bar{z})-\phi_{k+1}(\bar{x},\bar{y},\bar{z})\geq \theta_{k+1}+\sigma(2-\rho)\|\A^*x^{k+1}_{e}+\B^*y^k_{e}\|^2,
\end{equation}
where $\phi_k$ and $\theta_k$ are defined in (\ref{f}) and (\ref{t}), respectively.
\end{lemma}
\begin{proof}
On the one hand, setting $x:=x^{k+1}$ and $x':=x^k$ in (\ref{F2}), we get
$$
f(x^{k+1})\leq f(x^k)+\langle x^{k+1}-x^k,\nabla f(x^k) \rangle+
\frac{1}{2}\|x^{k+1}-x^k\|_{\widehat{\Sigma}_f}^2.
$$
On the other hand, setting $x':=x^k$ in (\ref{F1}), we get
$$
 f(x)\geq f(x^k)+\langle x-x^k,\nabla f(x^k) \rangle+\frac{1}{2}\|x^k-x\|_{\Sigma_f}^2,\quad \forall x\in\X.
$$
Combining the above two inequalites, we have
\begin{equation}\label{F}
f(x)-f(x^{k+1})-\langle x-x^{k+1},\nabla f(x^k) \rangle \geq
\frac{1}{2}\|x^k-x\|_{\Sigma_f}^2-\frac{1}{2}\|x^{k+1}-x^k\|_{\widehat{\Sigma}_f}^2, \qquad \forall x\in \X.
\end{equation}
From the first-order optimality condition of the $x$-subproblem in (\ref{mgadmm}), we obtain
\begin{equation}\label{fir}
\begin{aligned}
p(x)-p(x^{k+1})+\langle x-x^{k+1},\nabla f(x^k)+\A[z^k+\sigma(\A^*x^{k+1}+\B^*y^k-c)]
 +(\small{\widehat{\Sigma}_f}+\S)(x^{k+1}-x^k) \rangle\geq 0.\\[2mm]
\end{aligned}
\end{equation}
Adding both sides of (\ref{F}) and (\ref{fir}), it yields
\begin{equation}\label{FF}
\begin{aligned}
&\big(p(x)+f(x)\big)-\big(p(x^{k+1})+f(x^{k+1})\big)+\langle x-x^{k+1},\A[z^{k}+\sigma(\A^{*}x^{k+1}+\B^{*}y^{k}-c)]
+(\small{\widehat{\Sigma}_f}+\S)(x^{k+1}-x^k) \rangle \\[2mm]
\geq& \frac{1}{2}\|x^k-x\|_{\Sigma_f}^2-\frac{1}{2}\|x^{k+1}-x^k\|_{\widehat{\Sigma}_f}^2.
\end{aligned}
\end{equation}
Similarly, for any $y\in \Y $, we have
\begin{equation}\label{GG}
\begin{aligned}
&\big(q(y)+g(y)\big)-\big(q(y^{k+1})+g(y^{k+1})\big)\\[2mm]
&\quad +\langle y-y^{k+1},\mathcal{B}[z^{k}+\sigma\rho(\A^{*}x^{k+1}+\B^{*}y^{k+1}-c)
+\sigma(1-\rho)\B^{*}(y^{k+1}-y^k)]
+(\small{\widehat{\Sigma}_g}+\T)(y^{k+1}-y^k) \rangle \\[2mm]
&\geq \frac{1}{2}\|y^k-y\|_{\Sigma_g}^2-\frac{1}{2}\|y^{k+1}-y^k\|_{\widehat{\Sigma}_g}^2.
\end{aligned}
\end{equation}
Note that the third implementation in (\ref{mgadmm}) obeys the following equivalent form
$$
z^k+\sigma(\A^*x^{k+1}+\B^*y^k-c)
=z^{k+1}+\sigma(1-\rho)(\A^*x^{k+1}+\B^*y^k-c)+\sigma\B^*(y^k-y^{k+1}).
$$
Adding both sides of (\ref{FF}) and (\ref{GG}), and setting $x=\bar{x}$ and $y=\bar{y}$, we get
\begin{equation}\label{FG}
\begin{aligned}
&\Big(p(\bar{x})+f(\bar{x})+q(\bar{y})+g(\bar{y})\Big)-\Big(p(x^{k+1})+f(x^{k+1})+q(y^{k+1})+g(y^{k+1})\Big)\\[2mm]
&\quad -\Big\langle x^{k+1}_{e},\A z^{k+1}+\sigma(1-\rho)\A(\A^*x^{k+1}+\B^*y^k-c)+\sigma\A\B^*(y^k-y^{k+1})
+(\small{\widehat{\Sigma}_f}+\S)(x^{k+1}-x^k) \Big\rangle \\[2mm]
&\quad -\Big\langle y^{k+1}_{e},\B z^{k+1}+(\small{\widehat{\Sigma}_g}+\T)(y^{k+1}-y^k) \Big\rangle \\[2mm]
\geq& \frac{1}{2}\|x^k_{e}\|_{\Sigma_f}^2-\frac{1}{2}\|x^{k+1}-x^k\|_{\widehat{\Sigma}_f}^2
+\frac{1}{2}\|y^k_{e}\|_{\Sigma_g}^2-\frac{1}{2}\|y^{k+1}-y^k\|_{\widehat{\Sigma}_g}^2.
\end{aligned}
\end{equation}
Setting $x:=x^{k+1}$, $x':=\bar{x}$ in (\ref{F1}) and $y:=y^{k+1}$, $y':=\bar{y}$ in (\ref{G1}), we obtain
\begin{equation}\label{F3}
f(x^{k+1})- f(\bar{x})-\langle x^{k+1}_{e},\nabla f(\bar{x}) \rangle \geq \frac{1}{2}\|x^{k+1}_{e}\|_{\Sigma_f}^2,
\end{equation}
\begin{equation}\label{G3}
g(y^{k+1})- g(\bar{y})-\langle y^{k+1}_{e},\nabla g(\bar{y}) \rangle \geq \frac{1}{2}\|y^{k+1}_{e}\|_{\Sigma_g}^2.
\end{equation}
Adding both sides of (\ref{FG}), (\ref{F3}) and (\ref{G3}), and using the elementary inequality of (\ref{iec}), we have
\begin{equation}\label{FGG}
\begin{aligned}
&\Big(p(\bar{x})+q(\bar{y})\Big)-\Big(p(x^{k+1})+q(y^{k+1})\Big)\\[2mm]
&\quad -\Big\langle x^{k+1}_{e},\nabla f(\bar{x})+\A\bar{z}\Big\rangle-\Big\langle y^{k+1}_{e},\nabla g(\bar{y})+\B\bar{z}\Big\rangle+\Xi^{k}\\[2mm]
&\quad -\Big\langle x^{k+1}_{e},(\small{\widehat{\Sigma}_f}+\S)(x^{k+1}-x^k) \Big\rangle
-\Big\langle y^{k+1}_{e},(\small{\widehat{\Sigma}_g}+\T)(y^{k+1}-y^k) \Big\rangle \\[2mm]
\geq& \frac{1}{2}\|x^k_{e}\|_{\Sigma_f}^2+\frac{1}{2}\|x^{k+1}_{e}\|_{\Sigma_f}^2
+\frac{1}{2}\|y^k_{e}\|_{\Sigma_g}^2+\frac{1}{2}\|y^{k+1}_{e}\|_{\Sigma_g}^2 -\frac{1}{2}\|x^{k+1}-x^k\|_{\widehat{\Sigma}_f}^2-\frac{1}{2}\|y^{k+1}-y^k\|_{\widehat{\Sigma}_g}^2\\[2mm]
\geq& \frac{1}{4}\|x^{k+1}-x^k\|_{\Sigma_f}^2+\frac{1}{4}\|y^{k+1}-y^k\|_{\Sigma_g}^2
-\frac{1}{2}\|x^{k+1}-x^k\|_{\widehat{\Sigma}_f}^2-\frac{1}{2}\|y^{k+1}-y^k\|_{\widehat{\Sigma}_g}^2,
\end{aligned}
\end{equation}
where
$$
\Xi^k:=-\Big\langle x^{k+1}_{e},\A z^{k+1}_{e}+\sigma(1-\rho)\A(\A^*x^{k+1}+\B^*y^k-c)
+\sigma\A\B^{*}(y^{k}-y^{k+1})\Big\rangle-\Big\langle y^{k+1}_{e},\B z^{k+1}_{e} \Big\rangle.
$$
Setting $x:=x^{k+1}$ in the variational inequality (\ref{VA}), we have
\begin{equation}\label{VAR}
\begin{aligned}
&\Big(p(x^{k+1})+q(y^{k+1})\Big)-\Big(p(\bar{x})+q(\bar{y})\Big)+\Big\langle x^{k+1}_{e},\nabla f(\bar{x})+\mathcal{A}\bar{z} \Big\rangle +\Big \langle y^{k+1}_{e},\nabla g(\bar{y})+\mathcal{B}\bar{z} \Big\rangle \geq 0.
\end{aligned}
\end{equation}
Adding both sides of (\ref{FGG}) and (\ref{VAR}), we get
\begin{equation}\label{Q}
\begin{aligned}
&\Xi^k-\Big\langle x^{k+1}_{e},(\widehat{\Sigma}_f+\S)(x^{k+1}-x^k) \Big\rangle
-\Big\langle y^{k+1}_{e},(\widehat{\Sigma}_g+\T)(y^{k+1}-y^k)\Big \rangle\\[2mm]
\geq& \frac{1}{4}\|x^{k+1}-x^k\|_{\Sigma_f}^2+\frac{1}{4}\|y^{k+1}-y^k\|_{\Sigma_g}^2
-\frac{1}{2}\|x^{k+1}-x^k\|_{\widehat{\Sigma}_f}^2-\frac{1}{2}\|y^{k+1}-y^k\|_{\widehat{\Sigma}_g}^2.
\end{aligned}
\end{equation}
We note that the $\Xi^k$ can be reformulated as
\begin{equation}\label{QQ}
\begin{aligned}
\Xi^k&=-\Big\langle\A^*x^{k+1}_{e},z^{k+1}_{e}\Big\rangle
-\Big\langle\A^*x^{k+1}_{e},\sigma(1-\rho)(\A^*x^{k+1}+\B^*y^k-c)+\sigma\B^*(y^k-y^{k+1})\Big\rangle
-\Big\langle \B^*y^{k+1}_{e},z^{k+1}_{e}\Big\rangle\\[2mm]
&=-\Big\langle\A^*x^{k+1}_{e}+\B^*y^{k+1}_{e},z^{k+1}_{e}\Big\rangle
+\Big\langle\A^*x^{k+1}_{e},\sigma(\rho-1)(\A^*x^{k+1}+\B^*y^k-c)+\sigma\B^*(y^{k+1}-y^k)\Big\rangle,
\end{aligned}
\end{equation}
where
$$
\begin{aligned}
&\Big\langle\A^*x^{k+1}_{e}+\B^*y^{k+1}_{e},z^{k+1}_{e}\Big\rangle\\[2mm]
=&\Big\langle\A^*x^{k+1}_{e}+\B^*y^{k+1}_{e},z^{k+1}_{e}+\sigma(\rho-1)\B^*y^{k+1}_{e}\Big\rangle
-\sigma(\rho-1)\Big\langle\A^*x^{k+1}_{e}+\B^*y^{k+1}_{e},\B^*y^{k+1}_{e}\Big\rangle.
\end{aligned}
$$
Using the first identity in (\ref{id2}), we have
\begin{equation}\label{QQ1}
\begin{aligned}
&\sigma(\rho-1)\langle\A^*x^{k+1}_{e}+\B^*y^{k+1}_{e},\B^*y^{k+1}_{e}\rangle\\[2mm]
=&\frac{\sigma(\rho-1)}{2}\Big(\|\A^*x^{k+1}_{e}+\B^*y^{k+1}_{e}\|^2+\|\B^*y^{k+1}_{e}\|^2-\|\A^*x^{k+1}_{e}\|^2\Big).
\end{aligned}
\end{equation}
Substituing (\ref{l3}), (\ref{l4}) and (\ref{QQ1}) into (\ref{QQ}), we know that
\begin{equation}\label{QQQ}
\begin{aligned}
\Xi^k&=\frac{1}{2\sigma\rho}\big(\|z^k_{e}+\sigma(\rho-1)\B^*y^k_{e}\|^2
-\|z^{k+1}_{e}+\sigma(\rho-1)\B^*y^{k+1}_{e}\|^2\big)\\[2mm]
&\quad+\frac{\sigma(2-\rho)}{2}\big(\|\B^*y^k_{e}\|^2-\|\B^*y^{k+1}_{e}\|^2\big)
-\frac{\sigma(2-\rho)}{2}\|\A^*x^{k+1}_{e}+\B^*y^k_{e}\|^2.
\end{aligned}
\end{equation}
Applying the second identity in (\ref{id2}), we get
\begin{equation}\label{QC1}
\begin{aligned}
&\langle x^{k+1}_{e},(\small{\widehat{\Sigma}_f}+\S)(x^{k+1}-x^k) \rangle
=\frac{1}{2}\big(\|x^{k+1}_{e}\|_{\widehat{\Sigma}_f+\S}^2-\|x^k_{e}\|_{\widehat{\Sigma}_f+\S}^2
+\|x^{k+1}-x^k\|_{\widehat{\Sigma}_f+\S}^2\big),
\end{aligned}
\end{equation}
and
\begin{equation}\label{QC2}
\begin{aligned}
&\langle y^{k+1}_{e},(\small{\widehat{\Sigma}_g}+\T)(y^{k+1}-y^k) \rangle
=\frac{1}{2}\big(\|y^{k+1}_{e}\|_{\widehat{\Sigma}_g+\T}^2-\|y^k_{e}\|_{\widehat{\Sigma}_g+\T}^2
+\|y^{k+1}-y^k\|_{\widehat{\Sigma}_g+\T}^2\big).
\end{aligned}
\end{equation}
Substituing (\ref{QQQ}), (\ref{QC1}) and (\ref{QC2}) into (\ref{Q}), we obtain
\begin{equation}\label{O}
\begin{aligned}
&\frac{1}{2\sigma\rho}\big(\|z^k_{e}+\sigma(\rho-1)\B^*y^k_{e}\|^2
-\|z^{k+1}_{e}+\sigma(\rho-1)\B^*y^{k+1}_{e}\|^2\big)\\[2mm]
&\quad+\frac{\sigma(2-\rho)}{2}\big(\|\B^*y^k_{e}\|^2-\|\B^*y^{k+1}_{e}\|^2\big)
-\frac{\sigma(2-\rho)}{2}\|\A^*x^{k+1}_{e}+\B^*y^k_{e}\|^2\\[2mm]
&\quad+\frac{1}{2}\big(\|x^k_{e}\|_{\widehat{\Sigma}_f+\S}^2-\|x^{k+1}_{e}\|_{\widehat{\Sigma}_f+\S}^2
-\|x^{k+1}-x^k\|_{\widehat{\Sigma}_f+\S}^2\big)\\[2mm]
&\quad +\frac{1}{2}\big(\|y^k_{e}\|_{\widehat{\Sigma}_g+\T}^2-\|y^{k+1}_{e}\|_{\widehat{\Sigma}_g+\T}^2
-\|y^{k+1}-y^k\|_{\widehat{\Sigma}_g+\T}^2\big)\\[2mm]
\geq &\frac{1}{4}\|x^{k+1}-x^k\|_{\Sigma_f}^2+\frac{1}{4}\|y^{k+1}-y^k\|_{\Sigma_g}^2
-\frac{1}{2}\|x^{k+1}-x^k\|_{\widehat{\Sigma}_f}^2-\frac{1}{2}\|y^{k+1}-y^k\|_{\widehat{\Sigma}_g}^2.
\end{aligned}
\end{equation}
Using the notations in (\ref{f}) and (\ref{t}), we can get (\ref{in}) from (\ref{O}) immediately.
\end{proof}

\section{Convergence analysis}\label{section4}
In this section, we establish the convergence of algorithm MGADMM for solving problem (\ref{pro}).
For this purpose, we first give the following result.
\begin{lemma}\label{M}
Assume that $\frac{1}{2}\widehat{\Sigma}_g+\T\succeq0$. Let $\{(x^k, y^k, z^k )\}$ be the sequence generated by algorithm MGADMM. For any $\rho\in(0,2)$ and $k\geq1$, we have
\begin{equation}\label{MG}
\begin{aligned}
&\sigma(2-\rho)\|\A^*x^{k+1}+\B^*y^k-c\|^2\\[2mm]
\geq& \ \sigma(2-\rho)\|\A^*x^{k+1}+\B^*y^{k+1}-c\|^2+\frac{2-\rho}{\rho}\big(\zeta_{k+1}-\zeta_k\big)
+\frac{\sigma(2-\rho)^2}{\rho}\|\B^*(y^{k+1}-y^k)\|^2,
\end{aligned}
\end{equation}
where $\zeta_k$ is defined in (\ref{z}).
\end{lemma}
\begin{proof}
Note that
\begin{equation}\label{MA}
\begin{aligned}
&\sigma(2-\rho)\|\A^*x^{k+1}+\B^*y^k-c\|^2\\[2mm]
=&\sigma(2-\rho)\|\A^*x^{k+1}+\B^*y^{k+1}-c+\B^*(y^k-y^{k+1})\|^2\\[2mm]
=&\sigma(2-\rho)\|\A^*x^{k+1}+\B^*y^{k+1}-c\|^2+\sigma(2-\rho)\|\B^*(y^{k+1}-y^k)\|^2\\[2mm]
&\quad+2\sigma(2-\rho)\langle \B^*(y^k-y^{k+1}), \A^*x^{k+1}+\B^*y^{k+1}-c\rangle.
\end{aligned}
\end{equation}
By (\ref{zz}), we have
\begin{align}
&2\sigma(2-\rho)\langle \B^*(y^k-y^{k+1}), \A^*x^{k+1}+\B^*y^{k+1}-c\rangle\nonumber\\[1mm]
=&2\frac{(2-\rho)}{\rho}\langle \B^*(y^k-y^{k+1}), z^{k+1}-z^k+\sigma(\rho-1)\B^*(y^{k+1}-y^k)\rangle\nonumber\\[1mm]
=&2\frac{(2-\rho)}{\rho}\langle \B^*(y^k-y^{k+1}), z^{k+1}-z^k\rangle
-2\frac{\sigma(\rho-1)(2-\rho)}{\rho}\|\B^*(y^{k+1}-y^k)\|^2.\label{MAA}
\end{align}
The first-order optimality condition of the $y$-subproblem in (\ref{mgadmm}) shows that
\begin{equation}\label{FO}
\begin{aligned}
&-\nabla g(y^k)-\B z^{k+1}-(\widehat{\Sigma}_g+\T)(y^{k+1}-y^k)\in \partial q(y^{k+1}),\\[2mm]
&-\nabla g(y^{k-1})-\B z^k-(\widehat{\Sigma}_g+\T)(y^k-y^{k-1})\in \partial q(y^k).
\end{aligned}
\end{equation}
Since $\nabla g$ is assumed to be globally Lipschitz continuous, it is known from Clarke¡¯s Mean-Value Theorem \cite[Propositio 2.6.5]{CLA} that, there exists a self-adjoint and positive semidefinite linear operator $W^k\in\partial^2g(\omega^k)$ for a certain $\omega^k\in[y^{k-1},y^k]$ such that
\begin{equation}\label{MV}
\begin{aligned}
\nabla g(y^k)-\nabla g(y^{k-1})=W^k(y^k-y^{k-1}).
\end{aligned}
\end{equation}
Using the maximal monotonicity of $\partial q(\cdot)$ and (\ref{FO}), we get
$$
\Big\langle y^k-y^{k+1},\Big[-\nabla g(y^{k-1})-\B z^k-(\widehat{\Sigma}_g+\T)(y^k-y^{k-1})\Big]-\Big[-\nabla g(y^k)-\B z^{k+1}-(\widehat{\Sigma}_g+\T)(y^{k+1}-y^k)\Big]\Big\rangle \geq 0,
$$
which, together with (\ref{MV}), implies that
\begin{equation}\label{MAAA}
\begin{aligned}
&2\frac{(2-\rho)}{\rho}\Big\langle \B^*(y^k-y^{k+1}), z^{k+1}-z^k\Big\rangle
=2\frac{(2-\rho)}{\rho}\Big\langle y^k-y^{k+1}, \B(z^{k+1}-z^k)\Big\rangle\\[2mm]
\geq& 2\frac{(2-\rho)}{\rho}\Big\langle y^{k+1}-y^k,\nabla g(y^k)- \nabla g(y^{k-1}) \Big\rangle
-2\frac{(2-\rho)}{\rho}\Big\langle y^{k+1}-y^k,(\widehat{\Sigma}_g+\T)(y^k-y^{k-1}) \Big\rangle\\[2mm]
&\quad+2\frac{(2-\rho)}{\rho}\|y^{k+1}-y^k\|^2_{\widehat{\Sigma}_g+\T}\\[2mm]
=&2\frac{(2-\rho)}{\rho}\Big\langle y^{k+1}-y^k, W^k(y^k-y^{k-1}) \Big\rangle
-2\frac{(2-\rho)}{\rho}\Big\langle y^{k+1}-y^k,(\widehat{\Sigma}_g+\T)(y^k-y^{k-1}) \Big\rangle\\[2mm]
&\quad+2\frac{(2-\rho)}{\rho}\|y^{k+1}-y^k\|^2_{\widehat{\Sigma}_g+\T}\\[2mm]
=&2\frac{(2-\rho)}{\rho}\Big\langle y^{k+1}-y^k, (\widehat{\Sigma}_g+\T-W^k)(y^{k-1}-y^k)\Big\rangle
+2\frac{(2-\rho)}{\rho}\|y^{k+1}-y^k\|^2_{\widehat{\Sigma}_g+\T}.
\end{aligned}
\end{equation}
Applying the first identity in (\ref{id2}) and noting that $W^k\succeq 0$, we have
\begin{align}
&2\frac{(2-\rho)}{\rho}\big\langle y^{k+1}-y^k, (\widehat{\Sigma}_g+\T-W^k)(y^{k-1}-y^k)\big\rangle\nonumber\\[2mm]
=&\frac{2-\rho}{\rho}\|y^{k+1}-y^k\|^2_{\widehat{\Sigma}_g+\T-W^k}+\frac{2-\rho}{\rho}\|y^k-y^{k-1}\|^2_{\widehat{\Sigma}_g+\T-W^k}
-\frac{2-\rho}{\rho}\|y^{k+1}-y^{k-1}\|^2_{\widehat{\Sigma}_g+\T-W^k}\nonumber\\[2mm]
\geq&\frac{2-\rho}{\rho}\|y^{k+1}-y^k\|^2_{\widehat{\Sigma}_g+\T-W^k}+\frac{2-\rho}{\rho}\|y^k-y^{k-1}\|^2_{\widehat{\Sigma}_g+\T-W^k}
-\frac{2-\rho}{\rho}\|y^{k+1}-y^{k-1}\|^2_{\widehat{\Sigma}_g+\T-\frac{1}{2}W^k}\nonumber\\[2mm]
\geq&\frac{2-\rho}{\rho}\|y^{k+1}-y^k\|^2_{\widehat{\Sigma}_g+\T-W^k}
+\frac{2-\rho}{\rho}\|y^k-y^{k-1}\|^2_{\widehat{\Sigma}_g+\T-W^k}\nonumber\\[2mm]
&\quad-2\frac{(2-\rho)}{\rho}\|y^{k+1}-y^k\|^2_{\widehat{\Sigma}_g+\T-\frac{1}{2}W^k}
-2\frac{(2-\rho)}{\rho}\|y^k-y^{k-1}\|^2_{\widehat{\Sigma}_g+\T-\frac{1}{2}W^k}\nonumber\\[2mm]
=&-\frac{2-\rho}{\rho}\|y^{k+1}-y^k\|^2_{\widehat{\Sigma}_g+\T}
-\frac{2-\rho}{\rho}\|y^k-y^{k-1}\|^2_{\widehat{\Sigma}_g+\T},\label{ME}
\end{align}
where the second inequality is obtained from (\ref{iec}) and the fact (\ref{GG}) and
$$
\widehat{\Sigma}_g+\T-\frac{1}{2}W^k=\frac{1}{2}\widehat{\Sigma}_g+\T+\frac{1}{2}(\widehat{\Sigma}_g-W^k)
\succeq \frac{1}{2}\widehat{\Sigma}_g+\T \succeq 0.
$$
Substituting the inequality (\ref{ME}) into (\ref{MAAA}), we can estimate the term
\begin{equation}\label{MEE}
\begin{aligned}
&2\frac{(2-\rho)}{\rho}\langle \B^*(y^k-y^{k+1}), z^{k+1}-z^k\rangle\\[2mm]
\geq&-\frac{2-\rho}{\rho}\|y^{k+1}-y^k\|^2_{\widehat{\Sigma}_g+\T}
-\frac{2-\rho}{\rho}\|y^k-y^{k-1}\|^2_{\widehat{\Sigma}_g+\T}
+2\frac{(2-\rho)}{\rho}\|y^{k+1}-y^k\|^2_{\widehat{\Sigma}_g+\T}\\[2mm]
=&\frac{2-\rho}{\rho}\|y^{k+1}-y^k\|^2_{\widehat{\Sigma}_g+\T}
-\frac{2-\rho}{\rho}\|y^k-y^{k-1}\|^2_{\widehat{\Sigma}_g+\T},
\end{aligned}
\end{equation}
which, combining with (\ref{MA}), (\ref{MAA}) and the definition of $\zeta_k$, implies that (\ref{MG}) holds.
\end{proof}

Now we are ready to establish the global convergence of the algorithm MGADMM under some appropriate conditions.
\begin{theorem}\label{MG}
Let $\{(x^k, y^k, z^k )\}$ be the sequence generated by algorithm MGADMM and $(\bar{x},\bar{y},\bar{z})\in \W^*$. Let $\M_f$ and $\N_g$ be defined by (\ref{MN}). For each $k$, let $\phi_k$, $\theta_k$, $\vartheta_k$ and $\zeta_k$ be defined in (\ref{f}), (\ref{t}), (\ref{v}) and (\ref{z}), respectively. Then the following conclusions hold:\\
({\bf I}) For any $\rho\in(0,2)$ and $\beta\in(0,1/2)$, we have
\begin{equation}\label{i}
\begin{aligned}
&\Big(\phi_k(\bar{x},\bar{y},\bar{z})+\sigma\mu(2-\rho)\|\A^*x^k+\B^*y^k-c\|^2\Big)\\[2mm]
&\qquad\qquad-\Big(\phi_{k+1}(\bar{x},\bar{y},\bar{z})+\sigma\mu(2-\rho)\|\A^*x^{k+1}+\B^*y^{k+1}-c\|^2\Big)\\[2mm]
\geq& \Big(\frac{1}{\sigma\rho^2}\|z^{k+1}-z^k+\sigma(\rho-1)\B^*(y^{k+1}-y^k)\|^2\\[2mm]
&\qquad\qquad+\|x^{k+1}-x^k\|^2_{\widehat{\Sigma}_f+\S+\sigma\beta(2-\rho)\A\A^*}
+\|y^{k+1}-y^k\|^2_{\widehat{\Sigma}_g+\T+\sigma\beta(2-\rho)\B\B^*}\Big)\\[2mm]
&-\Big(\frac{1+2\mu-\mu\rho}{\sigma\rho^2}\|z^{k+1}-z^k+\sigma(\rho-1)\B^*(y^{k+1}-y^k)\|^2\\[2mm]
&\qquad\qquad+\|x^{k+1}-x^k\|^2_{\widehat{\Sigma}_f-\frac{1}{2}\Sigma_f}
+\|y^{k+1}-y^k\|^2_{\widehat{\Sigma}_g-\frac{1}{2}\Sigma_g+\sigma\beta(2-\rho)\B\B^*}\Big),
\end{aligned}
\end{equation}
where
\begin{equation}\label{m}
\begin{aligned}
\mu=\frac{\beta(1-\beta)}{1-2\beta}.
\end{aligned}
\end{equation}
Moreover, for some $\rho\in(0,2)$ and $\beta\in(0,1/2)$, we assume that
\begin{equation}\label{ias}
\begin{aligned}
\widehat{\Sigma}_f+\S+\sigma\beta(2-\rho)\A\A^* \succ 0  \qquad and \qquad \widehat{\Sigma}_g+\T+\sigma\beta(2-\rho)\B\B^* \succ 0,
\end{aligned}
\end{equation}
and that
\begin{equation}\label{ico}
\begin{aligned}
\sum\limits_{k=0}^{\infty}\Big(\|x^{k+1}-x^k\|^2_{\widehat{\Sigma}_f}
+\|y^{k+1}-y^k\|^2_{\widehat{\Sigma}_g+\sigma(2-\rho)\B\B^*}+\|\A^*x^{k+1}+\B^*y^{k+1}-c\|^2\Big)<+\infty.
\end{aligned}
\end{equation}
Then the sequence $\{(x^k,y^k)\}$ converges to an optimal solution of problem (\ref{pro}) and that $\{z^k\}$ converges to an optimal solution of the dual problem (\ref{dmp}) where $\varphi(x)\equiv p(x)+f(x)$ and $\phi(y)\equiv q(y)+g(y)$.

({\bf II}) Assume that
\begin{equation}\label{sigT}
\frac{1}{2}\widehat{\Sigma}_g+\T \succeq 0.
\end{equation}
Then, for any $\rho\in(0,2)$ and $\lambda\in(0,1]$, we have
\begin{equation}\label{ii}
\begin{aligned}
&\Big(\phi_k(\bar{x},\bar{y},\bar{z})+\sigma(1-\lambda)(2-\rho)\|\A^*x^k+\B^*y^k-c\|^2
+\frac{\lambda(2-\rho)}{\rho}\zeta_k\Big)\\[2mm]
&\quad-\Big(\phi_{k+1}(\bar{x},\bar{y},\bar{z})+\sigma(1-\lambda)(2-\rho)\|\A^*x^{k+1}+\B^*y^{k+1}-c\|^2
+\frac{\lambda(2-\rho)}{\rho}\zeta_{k+1}\Big)\\[2mm]
\geq& \ \vartheta_{k+1}+\sigma(2\lambda-1)(2-\rho)\|\A^*x^{k+1}+\B^*y^{k+1}-c\|^2.
\end{aligned}
\end{equation}
Morevoer, for some $\rho\in(0,2)$ and $\lambda\in(\frac{1}{2},1]$, we assume that
\begin{equation}\label{iimn}
\widehat{\Sigma}_f+\S\succeq 0, \qquad \M_f \succeq 0 \qquad and \qquad \N_g \succ 0,
\end{equation}
\begin{equation}\label{iixy}
\frac{1}{2}\Sigma_f+\S+\sigma(2-\rho)\A\A^*\succ0 \qquad and \qquad \frac{1}{2}\Sigma_g+\T+\sigma(2-\rho)\B\B^*\succ0.
\end{equation}
Then the sequence $\{(x^k,y^k)\}$ converges to an optimal solution of problem (\ref{pro}) and $\{z^k\}$ converges to an optimal solution of the dual problem (\ref{dmp}) where $\varphi(x)\equiv p(x)+f(x)$ and $\phi(y)\equiv q(y)+g(y)$.
\end{theorem}
\begin{proof}
({\bf I})
For a given $\mu>0$ defined in (\ref{m}), we have
\begin{align}
&\sigma(2-\rho)\|\A^*x^{k+1}+\B^*y^k-c\|^2\nonumber\\[2mm]
=&\sigma(2-\rho)\|\A^*x^k+\B^*y^k-c+\A^*(x^{k+1}-x^k)\|^2\nonumber\\[2mm]
=&\sigma(2-\rho)\|\A^*x^k+\B^*y^k-c\|^2+\sigma(2-\rho)\|\A^*(x^{k+1}-x^k)\|^2\nonumber\\[2mm]
&\quad+2\sigma(2-\rho)\langle \A^*(x^{k+1}-x^k), \A^*x^k+\B^*y^k-c\rangle\nonumber\\[2mm]
\geq&-\sigma\mu(2-\rho)\|\A^*x^k+\B^*y^k-c\|^2+\frac{\sigma\mu(2-\rho)}{1+\mu}\|\A^*(x^{k+1}-x^k)\|^2,\label{ip}
\end{align}
where the last inequality is from Cauchy-Schwarz inequality, that is
$$
\begin{aligned}
&2\sigma(2-\rho)\langle \A^*(x^{k+1}-x^k), \A^*x^k+\B^*y^k-c\rangle\\[2mm]
\geq& -\sigma(1+\mu)(2-\rho)\|\A^*x^k+\B^*y^k-c\|^2-\frac{\sigma(2-\rho)}{1+\mu}\|\A^*(x^{k+1}-x^k)\|^2.
\end{aligned}
$$
Recalling the definition of $\theta_{k}$ in (\ref{t}) and using the inequality of (\ref{ip}), we get
\begin{equation}\label{it}
\begin{aligned}
&\theta_{k+1}+\sigma(2-\rho)\|\A^*x^{k+1}+\B^*y^k-c\|^2\\[2mm]
\geq&-\frac{\mu(2-\rho)}{\sigma\rho^2}\|z^{k+1}-z^k+\sigma(\rho-1)\B^*(y^{k+1}-y^k)\|^2\\[2mm]
&\quad+\|x^{k+1}-x^k\|^2_{\frac{1}{2}\Sigma_f+\S+\frac{\mu}{1+\mu}\sigma(2-\rho)\A\A^*}
+\|y^{k+1}-y^k\|^2_{\frac{1}{2}\Sigma_g+\T}\\[2mm]
&\quad+\sigma\mu(2-\rho)\big(\|\A^*x^{k+1}+\B^*y^{k+1}-c\|^2-\|\A^*x^k+\B^*y^k-c\|^2\big).
\end{aligned}
\end{equation}
Noting the fact that
$$
\begin{aligned}
\frac{\mu}{1+\mu}=\frac{\frac{\beta(1-\beta)}{1-2\beta}}{1+\frac{\beta(1-\beta)}{1-2\beta}}
=\beta\frac{1-\beta}{1-\beta-\beta^2}>\beta,
\end{aligned}
$$
and the relationship that
$$
\begin{aligned}
&\frac{1}{2}\Sigma_f+\S+\frac{\mu}{1+\mu}\sigma(2-\rho)\A\A^*\succ (\widehat{\Sigma}_f+\S+\sigma\beta(2-\rho)\A\A^*)-(\widehat{\Sigma}_f-\frac{1}{2}\Sigma_f),\\[2mm]
&\frac{1}{2}\Sigma_g+\T = (\widehat{\Sigma}_g+\T+\sigma\beta(2-\rho)\B\B^*)-(\widehat{\Sigma}_g-\frac{1}{2}\Sigma_g+\sigma\beta(2-\rho)\B\B^*),\\[2mm]
\end{aligned}
$$
we obtain
\begin{equation}\label{itt}
\begin{aligned}
&-\frac{\mu(2-\rho)}{\sigma\rho^2}\|z^{k+1}-z^k+\sigma(\rho-1)\B^*(y^{k+1}-y^k)\|^2\\[2mm]
&\qquad\qquad+\|x^{k+1}-x^k\|^2_{\frac{1}{2}\Sigma_f+\S+\frac{\mu}{1+\mu}\sigma(2-\rho)\A\A^*}
+\|y^{k+1}-y^k\|^2_{\frac{1}{2}\Sigma_g+\T}\\[2mm]
&\geq \Big(\frac{1}{\sigma\rho^2}\|z^{k+1}-z^k+\sigma(\rho-1)\B^*(y^{k+1}-y^k)\|^2\\[2mm]
&\qquad\qquad +\|x^{k+1}-x^k\|^2_{\widehat{\Sigma}_f+\S+\sigma\beta(2-\rho)\A\A^*}
+\|y^{k+1}-y^k\|^2_{\widehat{\Sigma}_g+\T+\sigma\beta(2-\rho)\B\B^*}\Big) \\[2mm]
&-\Big(\frac{1+2\mu-\mu\rho}{\sigma\rho^2}\|z^{k+1}-z^k+\sigma(\rho-1)\B^*(y^{k+1}-y^k)\|^2\\[2mm]
&\qquad\qquad +\|x^{k+1}-x^k\|^2_{\widehat{\Sigma}_f-\frac{1}{2}\Sigma_f}
+\|y^{k+1}-y^k\|^2_{\widehat{\Sigma}_g-\frac{1}{2}\Sigma_g+\sigma\beta(2-\rho)\B\B^*}\Big).
\end{aligned}
\end{equation}
Substituting (\ref{itt}) and (\ref{it}) into (\ref{in}), we get the desired inequality (\ref{i}).

We now focus on the proving the desired convergence results.
Firstly, we show that there exits a convergence subsequence in the sequence $\{x^k,y^k,z^k\}$.
For this purpose, we denote
$$
\begin{aligned}
\delta_k:=&\frac{1}{\sigma\rho^2}\|z^{k+1}-z^k+\sigma(\rho-1)\B^*(y^{k+1}-y^k)\|^2\\[2mm]
& +\|x^{k+1}-x^k\|^2_{\widehat{\Sigma}_f+\S+\sigma\beta(2-\rho)\A\A^*}
+\|y^{k+1}-y^k\|^2_{\widehat{\Sigma}_g+\T+\sigma\beta(2-\rho)\B\B^*},\\[2mm]
\end{aligned}
$$
and
$$
\begin{aligned}
\xi_k:=&\frac{1+2\mu-\mu\rho}{\sigma\rho^2}\|z^{k+1}-z^k+\sigma(\rho-1)\B^*(y^{k+1}-y^k)\|^2\\[2mm]
& +\|x^{k+1}-x^k\|^2_{\widehat{\Sigma}_f-\frac{1}{2}\Sigma_f}
+\|y^{k+1}-y^k\|^2_{\widehat{\Sigma}_g-\frac{1}{2}\Sigma_g+\sigma\beta(2-\rho)\B\B^*}.\\[2mm]
\end{aligned}
$$
Using the notations, then (\ref{i}) can be rewritten equivalently as
\begin{equation}\label{ie}
\begin{aligned}
&\Big(\phi_k(\bar{x},\bar{y},\bar{z})+\sigma\mu(2-\rho)\|\A^*x^k+\B^*y^k-c\|^2\Big)\\[2mm]
&-\Big(\phi_{k+1}(\bar{x},\bar{y},\bar{z})+\sigma\mu(2-\rho)\|\A^*x^{k+1}+\B^*y^{k+1}-c\|^2\Big)
&\geq \delta_k-\xi_k.
\end{aligned}
\end{equation}
Under the assumptions of (\ref{ias}) and (\ref{ico}) with $\rho\in(0,2)$ and $\beta\in(0,1/2)$, we have the following inequalities by the using of the definitions of $\phi_{k}$, $\mu$ and Cauchy-Schwarz inequality, that is
\begin{equation}\label{ite}
\begin{aligned}
&\phi_{k+1}(\bar{x},\bar{y},\bar{z})+\sigma\mu(2-\rho)\|\A^*x^{k+1}+\B^*y^{k+1}-c\|^2\\[3mm]
=&(\sigma\rho)^{-1}\|z^{k+1}_{e}+\sigma(\rho-1)\B^*y^{k+1}_{e}\|^2
+\|x^{k+1}_{e}\|^2_{\widehat{\Sigma}_f+\S}+\|y^{k+1}_{e}\|^2_{\widehat{\Sigma}_g+\T+\sigma(2-\rho)\B\B^*}
+\sigma\mu(2-\rho)\|\A^*x^{k+1}_{e}+\B^*y^{k+1}_{e}\|^2\\[3mm]
\geq& (\sigma\rho)^{-1}\|z^{k+1}_{e}+\sigma(\rho-1)\B^*y^{k+1}_{e}\|^2
+\|x^{k+1}_{e}\|^2_{\widehat{\Sigma}_f+\S}+\|y^{k+1}_{e}\|^2_{\widehat{\Sigma}_g+\T+\sigma(2-\rho)\B\B^*}\\[3mm]
&\quad+\mu(2-\rho)(1-\frac{\beta}{1-\beta})\|x^{k+1}_{e}\|^2_{\sigma\A\A^*}
+\mu(2-\rho)(1-\frac{1-\beta}{\beta})\|y^{k+1}_{e}\|^2_{\sigma\B\B^*}\\[2mm]
=&(\sigma\rho)^{-1}\|z^{k+1}_{e}+\sigma(\rho-1)\B^*y^{k+1}_{e}\|^2
+\|x^{k+1}_{e}\|^2_{\widehat{\Sigma}_f+\S+\sigma\beta(2-\rho)\A\A^*}
+\|y^{k+1}_{e}\|^2_{\widehat{\Sigma}_g+\T+\sigma\beta(2-\rho)\B\B^*}.
\end{aligned}
\end{equation}
For the convergence results, we firstly show that $\{(x^k,y^k,z^k)\}$ is bounded. The assumption (\ref{ico}) indicates that $\sum_{k=0}^{\infty}\xi_k<+\infty$. Note that $\delta_k\geq 0$ from assumption (\ref{ias}). It follows from (\ref{ie}) and (\ref{ite}) that
\begin{equation}\label{iq}
\begin{aligned}
0&\leq \phi_{k+1}(\bar{x},\bar{y},\bar{z})+\sigma\mu(2-\rho)\|\A^*x^{k+1}+\B^*y^{k+1}-c\|^2\\[2mm]
&\leq \phi_k(\bar{x},\bar{y},\bar{z})+\sigma\mu(2-\rho)\|\A^*x^k+\B^*y^k-c\|^2+\xi_k\\
&\leq \phi_0(\bar{x},\bar{y},\bar{z})+\sigma\mu(2-\rho)\|\A^*x^0+\B^*y^0-c\|^2+\sum_{j=0}^{k}\xi_j\leq+\infty,
\end{aligned}
\end{equation}
which shows that $\{\phi_{k+1}(\bar{x},\bar{y},\bar{z})+\sigma\mu(2-\rho)\|\A^*x^{k+1}+\B^*y^{k+1}-c\|^2\}$ is bounded, and then implies that $\{\|z^{k+1}_{e}+\sigma(\rho-1)\B^*y^{k+1}_{e}\|\}$, $\{\|x^{k+1}_{e}\|_{\widehat{\Sigma}_f+\S+\sigma\beta(2-\rho)\A\A^*}\}$ and
$\{\|y^{k+1}_{e}\|_{\widehat{\Sigma}_g+\T+\sigma\beta(2-\rho)\B\B^*}\}$ are all bounded. Since
$\widehat{\Sigma}_f+\S+\sigma\beta(2-\rho)\A\A^*\succ 0$ and $\widehat{\Sigma}_g+\T+\sigma\beta(2-\rho)\B\B^*\succ 0$,
the sequence $\{\|x^{k+1}\|\}$ and $\{\|y^{k+1}\|\}$ are bounded. The boundedness
of $\{\|z^{k+1}_{e}+\sigma(\rho-1)\B^*y^{k+1}_{e}\|\}$ and $\{\|y^{k+1}\|\}$ further indicate that the sequence $\{\|z^{k+1}\|\}$ is also bounded. The above arguments have shown that $\{(x^k,y^k,z^k)\}$ is bounded.
Consequently, the sequence $\{(x^k,y^k,z^k)\}$ exists at least one subsequence $\{(x^{k_i},y^{k_i},z^{k_i})\}$ converging to a cluster point, say $(x^\infty,y^\infty,z^\infty)$.

Secondly, we prove that $(x^\infty,y^\infty)$ is an optimal solution to problem (\ref{pro}) and $z^\infty$ is the corresponding Lagrange multiplier.
For each $k\geq 0$, the inequality (\ref{ie}) is reorganized as the following form
\begin{equation}\label{ieq}
\begin{aligned}
&\delta_k\leq\Big(\phi_k(\bar{x},\bar{y},\bar{z})+\sigma\mu(2-\rho)\|\A^*x^k+\B^*y^k-c\|^2\Big)\\[2mm]
&\qquad-\Big(\phi_{k+1}(\bar{x},\bar{y},\bar{z})+\sigma\mu(2-\rho)\|\A^*x^{k+1}+\B^*y^{k+1}-c\|^2\Big)
+\xi_k.
\end{aligned}
\end{equation}
Summing both sides of (\ref{ieq}) from $j=0$ to $k$, it yields
\begin{equation}\label{iequ}
\begin{aligned}
\sum\limits_{j=0}^{k}\delta_j \leq& \Big(\phi_0(\bar{x},\bar{y},\bar{z})+\sigma\mu(2-\rho)\|\A^*x^0+\B^*y^0-c\|^2\Big)\\[-4mm]
&-\Big(\phi_{k+1}(\bar{x},\bar{y},\bar{z})+\sigma\mu(2-\rho)\|\A^*x^{k+1}+\B^*y^{k+1}-c\|^2\Big)
+\sum_{j=0}^{k}\xi_j,
\end{aligned}
\end{equation}
which implies that $\lim\limits_{k\rightarrow\infty}\delta_k=0$ by noting that $\sum\limits_{k=0}^{\infty}\xi_k<+\infty$.
Because $\widehat{\Sigma}_f+\S+\sigma\beta(2-\rho)\A\A^*\succ 0$ and $\widehat{\Sigma}_g+\T+\sigma\beta(2-\rho)\B\B^*\succ 0$ have been assumed in (\ref{ias}), from the definition of $\delta_k$, we get
\begin{equation}\label{lz}
\begin{aligned}
\lim\limits_{k\rightarrow\infty}\sigma\rho\|\A^*x^{k+1}+\B^*y^{k+1}-c\|
=\lim\limits_{k\rightarrow\infty}\|z^{k+1}-z^k+\sigma(\rho-1)\B^*(y^{k+1}-y^k)\|=0,
\end{aligned}
\end{equation}
\begin{equation}\label{lxy}
\begin{aligned}
\lim\limits_{k\rightarrow\infty}\|x^{k+1}-x^k\|=0 \qquad \mbox{and} \qquad \lim\limits_{k\rightarrow\infty}\|y^{k+1}-y^k\|=0.
\end{aligned}
\end{equation}
Using the first-order optimality condition of the $x$- and $y$-subproblem in (\ref{mgadmm}) on the subsequence $\{(x^{k_i},y^{k_i},z^{k_i})\}$, we know that
\begin{equation}\label{fooc}
\left\{
\begin{array}{l}
0\in \partial p(x^{{k_i}+1})+(\widehat{\Sigma}_f+\S)(x^{{k_i}+1}-x^{k_i})+\nabla f(x^{k_i})+
\A[z^{k_i}+\sigma(\A^*x^{{k_i}+1}+\B^*y^{k_i}-c)],\\[2mm]
0\in \partial q(y^{{k_i}+1})+(\widehat{\Sigma}_g+\T)(y^{{k_i}+1}-y^{k_i})+\nabla g(y^{k_i})
+\B[z^{k_i}+\sigma\rho(\A^*x^{{k_i}+1}+\B^*y^{{k_i}+1}-c)\\[2mm]
\qquad +\sigma(1-\rho)\B^*(y^{{k_i}+1}-y^{k_i})].
\end{array}
\right.
\end{equation}
Taking limits on both sides of (\ref{fooc}), and using (\ref{lz}) and (\ref{lxy}), we obtain
\begin{equation}\label{iKKT}
\begin{aligned}
0\in \partial p(x^\infty)+\nabla f(x^\infty)+\A{z^\infty}, \quad 0\in \partial q(y^\infty)+\nabla g(y^\infty)+\B{z^\infty}
\quad \mbox{and} \quad \A^*x^\infty+\B^*y^\infty=c,
\end{aligned}
\end{equation}
which indicates that $(x^\infty,y^\infty)$ is an optimal solution to problem (\ref{pro}) and $z^\infty$ is the corresponding Lagrange multiplier.

Finally, we need to show that $(x^\infty,y^\infty,z^\infty)$ is the unique limit point of the sequence $\{(x^k,y^k,z^k)\}$. Without loss of generality, we can let $(\bar{x},\bar{y},\bar{z})=(x^\infty,y^\infty,z^\infty)$. Consequently, the sequence $\{\phi_{k_i}(\bar{x},\bar{y},\bar{z})+\sigma\mu(2-\rho)\|\A^*x^{k_i}+\B^*y^{k_i}-c\|^2\}$ converges to 0 as $i \rightarrow\infty$. Hence, for any $k\geq k_i$, we know from (\ref{iq}) that
\begin{equation}\label{iqq}
\begin{aligned}
&\phi_{k+1}(\bar{x},\bar{y},\bar{z})+\sigma\mu(2-\rho)\|\A^*x^{k+1}+\B^*y^{k+1}-c\|^2\\[0mm]
\leq &\phi_{k_i}(\bar{x},\bar{y},\bar{z})+\sigma\mu(2-\rho)\|\A^*x^{k_i}+\B^*y^{k_i}-c\|^2+\sum_{j=k_i}^{k}\xi_j.
\end{aligned}
\end{equation}
Therefore, we get
\begin{equation}\label{lphi}
\begin{aligned}
&\lim \limits_{k\rightarrow\infty} \phi_{k+1}(\bar{x},\bar{y},\bar{z})+\sigma\mu(2-\rho)\|\A^*x^{k+1}+\B^*y^{k+1}-c\|^2=0,
\end{aligned}
\end{equation}
which, together with (\ref{ite}), implies that
\begin{align*}
&\lim \limits_{k\rightarrow\infty}\|z^{k+1}_{e}+\sigma(\rho-1)\B^*y^{k+1}_{e}\|=0,\\
&\lim \limits_{k\rightarrow\infty}\|x^{k+1}_{e}\|_{\widehat{\Sigma}_f+\S+\sigma\beta(2-\rho)\A\A^*}=0,\\
&\lim \limits_{k\rightarrow\infty}\|y^{k+1}_{e}\|_{\widehat{\Sigma}_g+\T+\sigma\beta(2-\rho)\B\B^*}=0.
\end{align*}
From the assumption
$\widehat{\Sigma}_f+\S+\sigma\beta(2-\rho)\A\A^*\succ 0$ and $\widehat{\Sigma}_g+\T+\sigma\beta(2-\rho)\B\B^*\succ 0$  stated in (\ref{ias}), we get $\lim \limits_{k\rightarrow\infty}x^k=\bar{x}$ and $\lim \limits_{k\rightarrow\infty}y^k=\bar{y}$.
Further, it is fact that
$$
\begin{aligned}
\|z^{k+1}_{e}\|
\leq \|z^{k+1}_{e}+\sigma(\rho-1)\B^*y^{k+1}_{e}\|+\sigma(\rho-1)\|\B^*y^{k+1}_{e}\|,
\end{aligned}
$$
and hence $\lim \limits_{k\rightarrow\infty}z^k=\bar{z}$. Therefore, the whole sequence $\{(x^k,y^k,z^k)\}$ converges to
$(\bar{x},\bar{y},\bar{z})$, which the unique limit of the sequence. This completes the proof of part ({\bf I}).

({\bf II})
Noting that $\frac{1}{2}\widehat{\Sigma}_g+\T \succeq 0$ is assumed, and using the Cauchy-Schwarz inequality, we have
\begin{equation}\label{iip}
\begin{aligned}
&\sigma(2-\rho)\|\A^*x^{k+1}+\B^*y^k-c\|^2\\[2mm]
=&\sigma(2-\rho)\|\A^*x^k+\B^*y^k-c+\A^*(x^{k+1}-x^k)\|^2\\[2mm]
=&\sigma(2-\rho)\|\A^*x^k+\B^*y^k-c\|^2+\sigma(2-\rho)\|\A^*(x^{k+1}-x^k)\|^2\\[2mm]
&\quad+2\sigma(2-\rho)\langle \A^*x^k+\B^*y^k-c, \A^*(x^{k+1}-x^k)\rangle\\[1mm]
\geq&-\sigma(2-\rho)\|\A^*x^k+\B^*y^k-c\|^2+\frac{1}{2}\sigma(2-\rho)\|\A^*(x^{k+1}-x^k)\|^2.
\end{aligned}
\end{equation}
By the above formula and using the definition of $\theta_{k+1}$, for any $\rho \in (0,2)$ and $\lambda \in (0,1]$, we get
\begin{equation}\label{iill}
\begin{aligned}
&(1-\lambda)[\theta_{k+1}+\sigma(2-\rho)\|\A^*x^{k+1}+\B^*y^k-c\|^2]\\[2mm]
\geq &(1-\lambda)\|x^{k+1}-x^k\|^2_{\frac{1}{2}\Sigma_f+\S+\frac{1}{2}\sigma(2-\rho)\A\A^*}
+(1-\lambda)\|y^{k+1}-y^k\|^2_{\frac{1}{2}\Sigma_g+\T}\\[2mm]
&\quad -\sigma(1-\lambda)(2-\rho)\|\A^*x^k+\B^*y^k-c\|^2.
\end{aligned}
\end{equation}
Using (\ref{MG}) and recalling the definition of $\theta_{k+1}$, we obtain, for any $\rho \in (0,2)$ and $\lambda \in (0,1]$, that
\begin{equation}\label{iil}
\begin{aligned}
&\lambda[\theta_{k+1}+\sigma(2-\rho)\|\A^*x^{k+1}+\B^*y^k-c\|^2]\\[2mm]
\geq& \lambda\|x^{k+1}-x^k\|^2_{\frac{1}{2}\Sigma_f+\S}+\lambda\|y^{k+1}-y^k\|^2_{\frac{1}{2}\Sigma_g+\T}
+\sigma\lambda(2-\rho)\|\A^*x^{k+1}+\B^*y^{k+1}-c\|^2\\[2mm]
&\quad+\frac{\lambda(2-\rho)}{\rho}(\zeta_{k+1}-\zeta_k)+\frac{\sigma\lambda(2-\rho)^2}{\rho}\|\B^*(y^{k+1}-y^k)\|^2.
\end{aligned}
\end{equation}
Adding both sides of (\ref{iill}) and (\ref{iil}), for any $\rho \in (0,2)$ and $\lambda \in (0,1]$, we know that
\begin{equation}\label{iiaa}
\begin{aligned}
&\theta_{k+1}+\sigma(2-\rho)\|\A^*x^{k+1}+\B^*y^k-c\|^2\\[2mm]
\geq& \|x^{k+1}-x^k\|^2_{\frac{1}{2}\Sigma_f+\S+\frac{1}{2}\sigma(1-\lambda)(2-\rho)\A\A^*}
+\|y^{k+1}-y^k\|^2_{\frac{1}{2}\Sigma_g+\T+\frac{\lambda(2-\rho)^2}{\rho}\sigma\B\B^*}
+\frac{\lambda(2-\rho)}{\rho}\big(\zeta_{k+1}-\zeta_k\big)\\[2mm]
& +\sigma(1-\lambda)(2-\rho)\big(\|\A^*x^{k+1}+\B^*y^{k+1}-c\|^2-\|\A^*x^k+\B^*y^k-c\|^2\big)\\[2mm]
& +\sigma(2\lambda-1)(2-\rho)\|\A^*x^{k+1}+\B^*y^{k+1}-c\|^2.
\end{aligned}
\end{equation}
Substituting this inequality into (\ref{in}), and using the notations in (\ref{MN}) and (\ref{v}), we get the desired inequality (\ref{ii}).

We now prove that $\{(x^k,y^k,z^k)\}$ is bounded. Note that $\rho \in (0,2)$ and $\lambda \in (1/2,1]$ is assumed. Since $\widehat{\Sigma}_g \succeq \Sigma_g$, we have $\widehat{\Sigma}_g+\T \succeq \frac{1}{2}\Sigma_g+\T \succeq 0$.
From $\M_f \succeq 0$ and $\N_g \succ 0$ which are assumed in (\ref{iimn}), we have $\phi_{k+1}(\bar{x},\bar{y},\bar{z})\geq 0 $, $\vartheta_{k+1}\geq 0$ and $\zeta_{k+1}\geq 0$. From (\ref{ii}), we see that the sequence $\{\phi_{k+1}(\bar{x},\bar{y},\bar{z})+\zeta_{k+1}\}$ is bounded, and
\begin{equation}\label{lv}
\lim\limits_{k\rightarrow\infty} \vartheta_{k+1}=0,
\end{equation}
and
\begin{equation}\label{lzz}
\lim\limits_{k\rightarrow\infty}(\sigma\rho)^{-1}\|z^{k+1}-z^k+\sigma(\rho-1)\B^*(y^{k+1}-y^k)\|
=\lim\limits_{k\rightarrow\infty}\|\A^*x^{k+1}+\B^*y^{k+1}-c\|=0.
\end{equation}
From (\ref{lv}), using the definition of $\vartheta_{k+1}$ in (\ref{v}) and the fact $\N_g \succ 0$ in (\ref{iimn}), we have
\begin{equation}\label{lmn}
\begin{aligned}
\lim\limits_{k\rightarrow\infty}\|x^{k+1}-x^k\|_{\M_f}=0 \quad \mbox{and} \quad
\lim\limits_{k\rightarrow\infty}\|y^{k+1}-y^k\|=0,
\end{aligned}
\end{equation}
and, as $k\rightarrow\infty$, that
$$
\|\A^*(x^{k+1}-x^k)\|\leq \|\A^*x^{k+1}+\B^*y^{k+1}-c\|+\|\A^*x^k+\B^*y^k-c\|+\|\B^*(y^{k+1}-y^k)\|\rightarrow 0.
$$
Using the above relation and (\ref{lmn}), we get
$$
\begin{aligned}
\lim\limits_{k\rightarrow\infty}\|x^{k+1}-x^k\|_{\frac{1}{2}\Sigma_f+\S+\sigma(2-\rho)\A\A^*}
=\lim\limits_{k\rightarrow\infty} \big(\|x^{k+1}-x^k\|_{\M_f}+\frac{1}{2}\sigma(1+\lambda)(2-\rho)\|\A^*(x^{k+1}-x^k)\|\big)=0.
\end{aligned}
$$
Since $\frac{1}{2}\Sigma_f+\S+\sigma(2-\rho)\A\A^*\succ 0$, we obtain
\begin{equation}\label{lxii}
\begin{aligned}
\lim\limits_{k\rightarrow\infty}\|x^{k+1}-x^k\|=0.
\end{aligned}
\end{equation}

Recall that the sequence $\phi_{k+1}(\bar{x},\bar{y},\bar{z})$ has been shown to be bounded. By the definition of $\phi_{k+1}(\bar{x},\bar{y},\bar{z})$, we obtain that the sequences $\{\|z^{k+1}_{e}+\sigma(\rho-1)\B^*y^{k+1}_{e}\|\}$, $\{\|x^{k+1}_{e}\|_{\widehat{\Sigma}_f+\S}\}$ and
$\{\|y^{k+1}_{e}\|_{\widehat{\Sigma}_g+\T+\sigma(2-\rho)\B\B^*}\}$ are all bounded.
Because $\widehat{\Sigma}_g+\T+\sigma(2-\rho)\B\B^* \succeq \frac{1}{2}\Sigma_g+\T+\sigma(2-\rho)\B\B^*\succ0$, then the sequence $\{\|y^{k+1}\|\}$ is bounded. The boundedness of $\{\|z^{k+1}_{e}+\sigma(\rho-1)\B^*y^{k+1}_{e}\|\}$
and $\{\|y^{k+1}\|\}$ further indicates that the sequence $\{\|z^{k+1}\|\}$ is also bounded.
Furthermore, by the fact
$$
\begin{aligned}
\|A^*x^{k+1}_{e}\| &\leq \|\A^*x^{k+1}+\B^*y^{k+1}-c\|+\|\B^*y^{k+1}_{e}\| \\[2mm]
&\leq (\sigma\rho)^{-1}\|z^{k+1}+\sigma(\rho-1)\B^*y^{k+1}\|
+(\sigma\rho)^{-1}\|z^k+\sigma(\rho-1)\B^*y^k\|+\|\B^*y^{k+1}_{e}\|,
\end{aligned}
$$
we can see that the sequence $\{\|\A^*x^{k+1}_{e}\|\}$ is bounded. Therefore,
$\{\|x^{k+1}_{e}\|_{\widehat{\Sigma}_f+\S+\sigma(2-\rho)\A\A^*}\}$ is bounded. It shows that the sequence
$\{\|x^{k+1}\|\}$ is also bounded because of $\widehat{\Sigma}_f+\S+\sigma(2-\rho)\A\A^* \succeq \frac{1}{2}\Sigma_f+\S+\sigma(2-\rho)\A\A^*\succ0$. Thus, the sequence $\{(x^k,y^k,z^k)\}$ is bounded.

Since the sequence $\{(x^k,y^k,z^k)\}$ is bounded, there admits at least one subsequence $\{(x^{k_i},y^{k_i},z^{k_i})\}$ converging to a cluster point, say $(x^\infty,y^\infty,z^\infty)$. We next prove that $(x^\infty,y^\infty)$ is an optimal solution to problem (\ref{pro}) and $z^\infty$ is the corresponding Lagrange multiplier.
Recall the first-order optimality condition of (\ref{mgadmm}), and along the subsequence $\{(x^{k_i},y^{k_i},z^{k_i})\}$, we know that
\begin{equation}\label{iifooc}
\left\{
\begin{array}{l}
0\in \partial p(x^{{k_i}+1})+(\widehat{\Sigma}_f+\S)(x^{{k_i}+1}-x^{k_i})+\nabla f(x^{k_i})+
\A[z^{k_i}+\sigma(\A^*x^{{k_i}+1}+\B^*y^{k_i}-c)],\\[2mm]
0\in \partial q(y^{{k_i}+1})+(\widehat{\Sigma}_g+\T)(y^{{k_i}+1}-y^{k_i})+\nabla g(y^{k_i})
+\B[z^{k_i}+\sigma\rho(\A^*x^{{k_i}+1}+\B^*y^{{k_i}+1}-c)\\[2mm]
\qquad +\sigma(1-\rho)\B^*(y^{{k_i}+1}-y^{k_i})].
\end{array}
\right.
\end{equation}
Taking limits on both sides of (\ref{iifooc}), using (\ref{lzz}), (\ref{lmn})and (\ref{lxii}), we have
\begin{equation}\label{iKKT}
\begin{aligned}
0\in \partial p(x^\infty)+\nabla f(x^\infty)+\A{z^\infty}, \quad 0\in \partial q(y^\infty)+\nabla g(y^\infty)+\B{z^\infty}
\quad \mbox{and} \quad \A^*x^\infty+\B^*y^\infty=c,
\end{aligned}
\end{equation}
which indicates that $(x^\infty,y^\infty)$ is an optimal solution to problem (\ref{pro}) and $z^\infty$ is the corresponding Lagrange multiplier.

Finally, we need to show that $(x^\infty,y^\infty,z^\infty)$ is the unique limit point of the sequence $\{(x^k,y^k,z^k)\}$. Without loss of generality, we can apply the inequality (\ref{iqq}) and (\ref{ii}), with
$(\bar{x},\bar{y},\bar{z})=(x^\infty,y^\infty,z^\infty)$ to know that
$$
\lim\limits_{k\rightarrow\infty} \phi_{k+1}(\bar{x},\bar{y},\bar{z})=0
\quad \mbox{and} \quad \lim\limits_{k\rightarrow\infty} \zeta_{k+1}=0,
$$
which indicates
$$
\lim \limits_{k\rightarrow\infty}\|z^{k+1}_{e}+\sigma(\rho-1)\B^*y^{k+1}_{e}\|=0,
$$
$$
\lim \limits_{k\rightarrow\infty}\|x^{k+1}_{e}\|_{\widehat{\Sigma}_f+\S}=0
\quad \mbox{and} \quad
\lim \limits_{k\rightarrow\infty}\|y^{k+1}_{e}\|_{\widehat{\Sigma}_g+\T+\sigma(2-\rho)\B\B^*}=0.
$$

Since $\widehat{\Sigma}_g+\T+\sigma(2-\rho)\B\B^*\succ 0$, we have $\lim \limits_{k\rightarrow\infty}y^k=\bar{y}$. Similar to case ({\bf I}), we can see that $\lim \limits_{k\rightarrow\infty}z^k=\bar{z}$. Using the fact that
$$
\begin{aligned}
\|A^*x^{k+1}_{e}\| &\leq \|\A^*x^{k+1}+\B^*y^{k+1}-c\|+\|\B^*y^{k+1}_{e}\| \\[2mm]
&\leq (\sigma\rho)^{-1}\|z^{k+1}_{e}+\sigma(\rho-1)\B^*y^{k+1}_{e}\|
+(\sigma\rho)^{-1}\|z^k_{e}+\sigma(\rho-1)\B^*y^k_{e}\|+\|\B^*y^{k+1}_{e}\| \rightarrow 0,
\end{aligned}
$$
we obtain
$$
\lim \limits_{k\rightarrow\infty}\|x^{k+1}_{e}\|_{\widehat{\Sigma}_f+\S+\sigma(2-\rho)\A\A^*}=0.
$$
It follows that $\lim \limits_{k\rightarrow\infty}x^k=\bar{x}$ by the positive definiteness of $\widehat{\Sigma}_f+\S+\sigma(2-\rho)\A\A^*$. Therefore, the whole sequence $\{(x^k,y^k,z^k)\}$ converges to
$(\bar{x},\bar{y},\bar{z})$, i.e., the unique limit of the sequence.
\end{proof}

\section{Numerical experiments}\label{section6}

In this section, we construct a series of numerical experiments using
simulated convex composite optimization problems and sparse inverse covariance matrix estimation problems
to demonstrate the feasibility and efficiency of the algorithm MGADMM.
Besides, we also test against the state-of-the-art algorithm sPADMM of Li et al. \cite{LI} and sGS-ADMM of Li \& Xiao \cite{LIXIAO} to evaluate the algorithms' performance.
All the experiments are performed with Microsoft Windows 10 and MATLAB R2018a, and run on a desktop PC with an Intel Core i7-8565 CPU at 1.80 GHz and 8 GB of memory.

\subsection{Test on simulated convex composite optimization problems}
In this test, we focus on the following convex composite quadratic optimization problem which has been considered in \cite{LI}:
\begin{equation}\label{pro1}
\min_{x \in \mathbb{R}^{n}, y\in \mathbb{R}^{m}}  \Big \{\frac{1}{2}\langle x, Qx \rangle-\langle b,x \rangle+\frac{\gamma}{2}\|\Pi_{\mathbb{R}_{+}^{m}}(D(d-Hx))\|^2+\mu \|x\|_{1}+\delta_{\mathbb{R}^{m}_{+}}(y) \ \big | \ Hx+y=c\Big \},
\end{equation}
where $Q\in \mathbb{R}^{n\times n}$ is a symmetric and positive semidefinite matrix (may not be positive definite), $\Pi_{\mathbb{R}_{+}^{m}}(\cdot)$ denotes the projection onto $\mathbb{R}^{m}_{+}$, $\|x\|_1:=\sum_{i=1}^{n}|x_i|$, and $\delta_{\mathbb{R}^{m}_{+}}(\cdot)$ is an indicator function on $\mathbb{R}^{m}_{+}$. $H\in \mathbb{R}^{m\times n}$, $b\in \mathbb{R}^n$, $c\in \mathbb{R}^{m}$, $d (\leq c) \in \mathbb{R}^{m}$, $\gamma\geq0$ and $\mu>0$ are given data. In addition, $D$ is a positive definite diagonal matrix chosen to normalize each row of $H$ to have the unit norm. Clearly, model (\ref{pro1}) is in the form of (\ref{pro}) if taking
\begin{align*}
&p(x):=\mu \|x\|_1, \quad f(x):=\frac{1}{2}\langle x, Qx\rangle-\langle b,x\rangle+\frac{\gamma}{2}\|\Pi_{\mathbb{R}_{+}^{m}}(D(d-Hx))\|^2,\\ &q(y):=\delta_{\mathbb{R}^{m}_{+}}(y), \quad g(y):= 0, \quad \A^*:=H, \quad \B^*:=I.
\end{align*}
The associated KKT system to problems (\ref{pro1}) is with the following form:
$$
Hx+y-c=0, \quad \nabla f(x)+H^* z+v=0, \quad y\geq 0, \quad z\geq 0, \quad y\circ z=0, \quad v \in \partial \mu \|x\|_1,
$$
where $\circ$ denotes the elementwise product, $z$ and $v$ denote the dual variables.

In this part, we employ MGADMM with $\rho=1.9$ and majorized sPADMM with $\tau=1.618$  to solve the problem (\ref{pro1}), respectively.
For the sake of fairness, we choose the same proximal terms and operators as \cite{LI}, i.e.,
\begin{align*}
&\widehat{\Sigma}_{f}=Q+\gamma H^{*}D^2 H, \quad \Sigma_f=Q, \quad \widehat{\Sigma}_g=0=\Sigma_g,\\
&\S=\lambda_{\max}(\widehat{\Sigma}_{f}+\sigma H^* H) I-(\hat{\Sigma}_{f}+\sigma H^* H), \quad \T=0.
\end{align*}
With these choices, it is easy to verify that the conditions (\ref{ias}), (\ref{sigT}), (\ref{iimn}) and (\ref{iixy}) in Theorem \ref{MG} are satisfied, and hence, the convergence of MGADMM is guaranteed.

We terminate the iterative process if the KKT residual is not greater than a given small tolerance, that is
$$
Res:=\max\Big\{\frac{\|Hx^{k+1}+y^{k+1}-c\|}{1+\|c\|},\quad \frac{\|\nabla f(x^{k+1})+H^*z^{k+1}+v^{k+1}\|}{1+\|b\|}\Big\}\leq 1e-5.
$$
Besides, we choose the zero vector as an initial value of each variable, and take the parameters' values as $\mu=5 \sqrt{n}$ and $\rho=1.9$.
The penalty parameter is initialized as $\sigma=0.8$ and updated dynamically at the whole iterative process.

Using these notations, we can observe that, while MGADMM is used to solve the problem (\ref{pro1}), each subproblem may admit closed-form solutions.
And then, the Step 1 of MGADMM can be states as follows. In the algorithm, the symbol $\B_{\infty}^{(l)}$ denotes a $\ell_{\infty}$-norm ball with radius $r>0$, i.e.,  $\B^{(l)}_{\infty}:=\{x:\|x\|_\infty\leq 1\}$, and the symbol $\Pi_{\C}(\cdot)$ represents the metric projection onto a closed convex set $\C$.

\begin{algorithm}
\noindent
{\bf Step 1 of MGADMM:}
\vskip 1.0mm \hrule \vskip 1mm
\noindent
\textbf{Step 1.1} Compute
$$
x^{k+1}:=r^k-\Pi_{B_{\infty}^{(\frac{\mu}{\omega_1})}}(r^k),
$$
where $\omega_1=\lambda_{\max}(\widehat{\Sigma}_{f}+\sigma H^* H)$ and $r^k=x^k-\frac{1}{\omega_1}\big[\nabla f(x^k)+\sigma H^*(Hx^k+y^k-c+\frac{z^k}{\sigma}) \big]$.\\
\textbf{Step 1.2} Compute
$$
y^{k+1}:=\Pi_{\mathbb{R}_{+}^m}\Big((1-\rho)y^k+\rho c-\rho Hx^{k+1}-\frac{z^k}{\sigma}\Big).
$$
\textbf{Step 1.3} Compute
$$
z^{k+1}:=z^k+\sigma\Big(\rho Hx^{k+1}-(1-\rho)y^k+y^{k+1}-\rho c\Big).
$$
\end{algorithm}

In this test, we choose different pairs of $(m,n)$ to generate the data used in (\ref{pro1}) and the numerical results with two kinds of $\gamma$ are listed in Tables \ref{tab11}-\ref{tab12}, which include the number of iterations (Iter), the computing time (Time), the final KKT residuals of the solution (Res).
From these tables, we find that MGADMM can always use slightly fewer iterations than sPADMM with an competitive computing time to achieve similar accuracy solutions.

\begin{table}[h]
\centering
\caption{Numerical results of sPADMM and MGADMM with $\gamma=0$.}
\setlength{\tabcolsep}{4mm}{
\begin{tabular}{|cc||c|c|c|c|c|c|}
\hline
\multicolumn{2}{|c||}{} & \multicolumn{3}{c|}{sPADMM} & \multicolumn{3}{c|}{MGADMM}   \\
\hline
m&   n& Iter & Time & Res & Iter & Time & Res         \\
\hline
500  & 200   & 5966 & 1.9 & 9.98e-6 & 5032  & 1.5 & 9.98e-6\\
500  & 500   & 794 & 1.3 & 9.96e-6 & 747  & 1.2 & 9.95e-6\\
500  & 1000  & 628  & 2.1 & 9.94e-6 & 579   & 2.0 & 9.98e-6\\
1000 & 500   & 1354 & 4.1 & 9.97e-6 & 1179  & 3.4 & 9.89e-6\\
1000 & 1000  & 758  & 4.3 & 9.89e-6 & 696   & 3.9 & 9.97e-6\\
1000 & 2000  & 630  & 8.3 & 9.91e-6 & 608   & 7.7 & 9.93e-6\\
2000 & 1000  & 3088 & 32.6 & 9.99e-6 & 2573  & 25.7 & 9.99e-6\\
2000 & 2000  & 778  & 17.6 & 9.88e-6 & 720   & 15.4 & 9.96e-6\\
2000 & 4000  & 587  & 30.9 & 9.95e-6 & 575  & 30.2 & 9.97e-6\\
4000 & 2000  & 1816 & 74.6 & 9.99e-6 & 1553 & 63.1 & 9.99e-6\\
4000 & 4000  & 708 & 68.7 & 9.96e-6 & 683  & 64.4 & 9.99e-6\\
8000 & 4000  & 1334 & 249.9 & 9.99e-6 & 1103 & 207.6 & 9.99e-6\\
8000 & 8000  & 793  & 356.8 & 9.97e-6 & 779  & 337.1 & 9.97e-6\\
\hline
\end{tabular}}
\label{tab11}
\end{table}

\begin{table}[h]
\centering
\caption{Numerical results of majorized sPADMM and MGADMM with $\gamma=2 \mu$.}
\setlength{\tabcolsep}{4mm}{
\begin{tabular}{|cc||c|c|c|c|c|c|}
\hline
\multicolumn{2}{|c||}{} & \multicolumn{3}{c|}{sPADMM} & \multicolumn{3}{c|}{MGADMM}   \\
\hline
m&   n& Iter & Time & Res & Iter & Time & Res         \\
\hline
500  & 200   & 7934 & 2.5 & 9.99e-6 & 6791  & 1.8 & 9.98e-6\\
500  & 500   & 480 & 0.8 & 9.99e-6 & 450  & 0.7 & 9.80e-6\\
500  & 1000  & 458  & 1.5 & 9.99e-6 & 445   & 1.4 & 9.92e-6\\
1000 & 500   & 1715 & 5.1 & 9.99e-6 & 1423  & 3.8 & 9.99e-6\\
1000 & 1000  & 430  & 2.4 & 9.96e-6 & 409   & 2.2 & 9.94e-6\\
1000 & 2000  & 478  & 6.1 & 9.88e-6 & 451   & 5.6 & 9.94e-6\\
2000 & 1000  & 1354 & 13.7 & 9.99e-6 & 1228  & 12.4 & 9.99e-6\\
2000 & 2000  & 341  & 7.8 & 9.88e-6 & 325   & 7.3 & 9.85e-6\\
2000 & 4000  & 406  & 21.6 & 9.95e-6 & 403  & 21.0 & 9.95e-6\\
4000 & 2000  & 1214 & 50.6 & 9.99e-6 & 995 & 40.6 & 9.99e-6\\
4000 & 4000  & 312  & 33.1 & 9.96e-6 & 306  & 31.8 & 9.93e-6\\
8000 & 4000  & 1028 & 195.3 & 9.99e-6 & 841 & 159.1 & 9.97e-6\\
8000 & 8000  & 315  & 172.6& 9.91e-6 & 314  & 166.2 & 9.98e-6\\
\hline
\end{tabular}}
\label{tab12}
\end{table}

\subsection{Test on sparse inverse covariance matrix estimation problems}

In this part, we further evaluate the efficiency of MGADMM by the using of a typical inverse covariance matrix estimation problem in the field of statistics.
Besides, we also test against the recent algorithm sGS-ADMM of Li \& Xiao \cite{LIXIAO} for performance comparison.
Let $\mathcal{S}^p$, $\mathcal{S}^p_{+}$ be the sets of all $p\times p$ symmetric, positive semi-definite matrices, respectively.
As in \cite{LIXIAO,YANG}, we consider the following log-determinant optimization problem with group lasso regularization, that is
\begin{equation}\label{GMLE}
\min_{X\in \S^{p}_{+}}\big\{\langle S, X\rangle-\log\det X+\sum\limits_{l=1}^{r}\omega_{l}\|X_{\{g_{l}\}}\|_{\#} \; | \;
\H X=b \big\},
\end{equation}
where $S$ is a known sample covariance matrix, $r$ is the number of groups, $\omega_{l}>0$ is an adaptive turn parameter, $X_{\{g_{l}\}}$ is a vector constructed by the components of $X$ with index set $g_{l}$, $\|\cdot\|_{\#} $ is the $\l_{\#}$-norm with popular choices $\#=\infty$, $1$, $2$, $\H: \S^p\rightarrow \mathbb{R}^m $ is a generic linear mapping and $b\in \mathbb{R}^m$.

For convenience, we let $\P_{l}$ be an operation such that $\P_{l} X=X_{\{g_{l}\}}$ for each group $l$, and then denote $\P:=[\P_1;\P_2;...;\P_r]$.
It is not hard to deduce that the Lagrangian dual of the problem (\ref{GMLE}) obeys the following form with a triple of dual variables $u$, $v$ and $Z$, i.e.,
\begin{equation}\label{DUAL}
\begin{array}{ll}
\min\limits_{v,u,Z} & \delta_{\B_{q}^{(w_l)}}(v)-\langle u, b\rangle-\log\det Z+\delta_{\S_{++}^{p}}(Z)\\
\text{s.t.}         & \H^*u+\P^*v+Z=S,
\end{array}
\end{equation}
where $\B_{q}^{(w_l)}:=\{v\in \mathbb{R}^s: \|v_l\|_q\leq w_l, l=1,2,\ldots,r\}$, $1/\#+1/q=1$ and $1/\infty=0$ is assumed.
Obviously, the dual problem (\ref{DUAL}) can be expressed as the form of (\ref{pro}) if we denoting
$$
p(x):=\delta_{B_{q}^{w_l}}(v), \quad f(x):=-\langle u, b\rangle, \quad q(y):=\delta_{\S_{++}^{p}}(Z), \quad g(y):=-\log\det Z,
$$
with
$$
x:=(u;v),
\quad y:=Z, \quad \A^{*}:=\left(
                \begin{array}{cc}
                  \P^* & \H^* \\
                \end{array}
              \right), \quad \B^*:=\I.
$$
From the corresponding KKT system, it is reasonable to set the stopping criterion as
\begin{equation}\label{resss}
Res:=\{R_D,R_P,R_G\}<1e-3,
\end{equation}
where
$$
R_{P}:=\frac{\|\H X-b\|}{1+\|b\|}, \quad R_{D}:=\max_l\Big\{\frac{\|\H^*u+\P^*v+Z-S\|}{1+\|S\|}, \frac{\|v_l\|_q-w_l}{w_l}\Big\}, \quad R_G:=\frac{|\text{pobj-dobj}|}{1+|\text{pobj}|+|\text{dobj}|},
$$
with
$$
\text{pobj}:=\langle S,X\rangle-\logdet X+\sum_{l=1}^{r}w_{l}\|X_{\{g_{l}\}}\|_{\#}, \quad\text{dobj}:=n+\logdet Z+\langle u,b \rangle.
$$
In this test, we apply majorized sGS-ADMM with $\tau=1.618$ and MGADMM with $\rho=1.99$ to solve the dual problem (\ref{DUAL}).
The proximal terms involved at the subproblems are chosen as
$$
\S:=\diag(\S_1,\S_2), \quad \S_1:=\sigma(\lambda_1-1)\I, \quad \S_2:=\sigma(\lambda_2 \I-\H\H^*), \quad \T:=0,
$$
where $\lambda_1=1.01$ and $\lambda_2=1.01\lambda_{\max}(\H\H^*)$.
Moreover, the self-adjoint and positive semidefinite operators $\Sigma_f$, $\widehat{\Sigma}_f$, $\Sigma_g$ and $\widehat{\Sigma}_g$ used at the majorization technique are chosen as
$$
\Sigma_f\equiv\widehat{\Sigma}_f\equiv 0, \quad \Sigma_g:=0, \quad \widehat{\Sigma}_g:=\sigma \xi \I,
$$
where $\xi>0$ is a positive constant that enables $\widehat{\Sigma}_g$ such that the inequality (\ref{G2}) holds.
Here, we set $\xi=0.01$, which has been illustrated to be a suitable choice in the experiments' preparations.

By the using of these notations, the majorized function of $g(y)$, or $-\log\det Z$, has the following explicit form at a given $Z'$, i.e.,
$$
\hat{g}(y,y'):=-\log\det Z'+\langle Z-Z', -(Z')^{-1}\rangle+\frac{\sigma \xi}{2}\|Z-Z'\|^2,
$$
and hence, the $y$-subproblem in Step 1 of MGADMM reduces to
\begin{align*}
Z^{k+1}:=&\argmin\limits_{Z\in \mathbb{R}^{p\times p}} \ \Big\{\delta_{\S_{++}^{p}}(Z)+\langle Z,-(Z^k)^{-1}\rangle+\frac{\sigma \xi}{2}\|Z-Z^k\|^2+\langle X^k,\H^*u^{k+1}+\P^*v^{k+1}+Z-S \rangle\\
&\qquad\qquad+\frac{1}{2}\|Z-Z^k\|_{\T}^2+\frac{\sigma}{2}\|\rho \H^*u^{k+1}+\rho \P^*v^{k+1}-(1-\rho)Z^k+Z-\rho S\|^2\Big\},\\
=&\Phi_{\frac{1}{\sigma}}\Big(Z^k-\frac{\rho}{\xi+1}\big(\H^*u^{k+1}+\P^*v^{k+1}+Z^k-S+\frac{1}{\sigma \rho}(X^k-(Z^k)^{-1})\big)\Big),
\end{align*}
where $\Phi_{\cdot}(\cdot)$ is a linear operator which is used to guarantee the positive definiteness of matrix $Z$, see e.g., \cite[Proposition 2.2]{LIXIAO}. On the other hand, noting that $x:=(u;v)$ and $v$ is restricted in a closed convex set $\B_{q}^{(w_l)}$, it is reasonable to use the sGS iteration to split the $x$-subproblem into a triple of smaller subproblem in a manner of $u\rightarrow v\rightarrow u$.
Additionally, from the sGS decomposition theorem of Li et al. \cite{SGSTH}, it is easy to deduce that this type of iteration can be treated together in the sense of adding an semi-proximal term.
Subsequently, the conditions (\ref{ias}), (\ref{sigT}), (\ref{iimn}) and (\ref{iixy}) in
Theorem \ref{MG} are satisfied, so the convergence of MGADMM is guaranteed.

Taking everything together, we can state the main steps of MGADMM as follows:

\begin{algorithm}
\noindent
{\bf Step 1 of MGADMM:}
\vskip 1.0mm \hrule \vskip 1mm
\noindent
\textbf{Step 1.1} Compute
$$
\left\{
\begin{array}{l}
\label{u1}\tilde{u}^k:=u^k+\frac{b}{\sigma \lambda_2}-\frac{1}{\lambda_2}\H \Big(\H^*u^k+\P^*v^k+Z^{k}-S+\frac{1}{\sigma}X^k\Big),\\[2mm]
\label{v}v^{k+1}:=\Pi_{\B_q^{(w_l)}}\Big(v^k-\frac{1}{\lambda_1}\P\big(\H^*\tilde{u}^k+\P^*v^k+Z^k-S+\frac{1}{\sigma}X^k\big)\Big),\\[2mm]
\label{u2}u^{k+1}:=u^k+\frac{b}{\sigma \lambda_2}-\frac{1}{\lambda_2}\H \Big(\H^*u^k+\P^*v^{k+1}+Z^{k}-S+\frac{1}{\sigma}X^k\Big).
\end{array}
\right.
$$
\textbf{Step 1.2}Compute
$$
\label{z}Z^{k+1}:=\Phi_{\frac{1}{\sigma}}\Big(Z^k-\frac{\rho}{\xi+1}\big(\H^*u^{k+1}+\P^*v^{k+1}+Z^k-S+\frac{1}{\sigma \rho}(X^k-(Z^k)^{-1})\big)\Big).
$$
\textbf{Step 1.3} Update
$$
X^{k+1}:=X^k+\sigma(\rho \H^*u^{k+1}+\rho \P^*v^{k+1}-(1-\rho)Z^k+Z^{k+1}-\rho S).
$$
\end{algorithm}

Similar to the technique in \cite{LIXIAO}, we initialize the penalty parameter $\sigma_0=1$ and adjust it frequently at the iterative process.
In each test, we start the algorithm from $(v^0,u^0,Z^0,X^0)=(0,0,\I,\I)$ and terminate if KKT residual (\ref{resss}) is sufficiently small.
The comparison results of MGADMM and sGS-ADMM with with $\#=1$, $2$, and $\infty$ are shown in Table \ref{tab2}.
From the data in the table, we can see that for the three types of norm considered, MGADMM always requires fewer number of iterations than sGS-ADMM to achieve competitive accuracy solutions.

\begin{table}[h]
\centering\caption{Numerical results of majorized sGS-ADMM and MGADMM.}
\setlength{\tabcolsep}{4mm}{
\begin{tabular}{|c|c||c|c|c|c|c|c|}
\hline
\multicolumn{2}{|c||}{} & \multicolumn{3}{c|}{sGS-ADMM} & \multicolumn{3}{c|}{MGADMM}   \\
\hline
$\#$& $p \ | \ m$& Iter & Time & Res & Iter & Time & Res         \\
\hline
\multirow{3}{*}{1} & 200 $|$ 8812 & 28  &0.3 &3.88e-4 & 27 &0.2 &7.57e-4\\
\cline{2-8}
& 500 $|$ 55072 & 31 & 1.7& 1.22e-4& 30 &1.7 &6.51e-4\\
\cline{2-8}
& 1000 $|$ 222685 & 32 & 8.9 & 2.19e-4 & 31 & 8.9& 9.31e-4\\
\hline
\multirow{3}{*}{2} & 200 $|$ 8812 & 33 & 0.4&4.76e-4  & 32&0.4 &5.96e-4\\
\cline{2-8}
& 500 $|$ 55072 & 38 & 2.5 &4.18e-4 & 35 &2.3 &9.07e-4\\
\cline{2-8}
& 1000 $|$ 222685 & 57& 22.1 &5.37e-4  & 52&20.22 &9.33e-4\\
\hline
\multirow{3}{*}{$\infty$} & 200 $|$ 8812 & 82 & 3.0 &9.82e-4 & 72 &2.6 &9.11e-4\\
\cline{2-8}
& 500 $|$ 55072 &  264 & 62.1 & 5.61e-4& 240 & 59.1 &8.81e-4\\
\cline{2-8}
& 1000 $|$ 222685 &  841 & 947.3&9.27e-4 & 616 & 652.9& 9.45e-4 \\
\hline
\end{tabular}}
\label{tab2}
\end{table}

Overall, these limited experiments demonstrate that MGADMM is competitive with sPADMM on a kind of simulated convex composite optimization problems, and
performs better than sGS-ADMM on a type of sparse inverse covariance matrix estimation problems.
Therefore, we conclude that, by the using of the majorization technique, the generalized ADMM of Eckstein \& Bertsekas \cite{ECK} can be successfully extended to solve
the convex composite optimization problem (\ref{pro}) and always performs better than the algorithms compared with a suitable relaxation factor.

\section{Conclusions}\label{section7}

In optimization literature, it is well-known that the classical ADMM with an unit step-length is equivalent to DRs method, and that DRs is an instance of the proximal point algorithm with a special splitting operator.
It is also known that the generalized-ADMM is a variant of DRs according to using a relaxation factor to achieve faster convergence.
However, when the ADMM is directly employed to solve problem (\ref{pro}), the structure of nonsmooth+smooth hidden in the subproblems can not be utilized.
In this paper, we attempted to remedy this defect, and showed that when the majorization technique is implemented, the corresponding subproblem has nonsmooth+quadratic structure and hence it can be decomposed into some smaller problems and solved efficiently.
The paper was concentrate on the convergence analysis and theoretically proved that our algorithm converges globally under some appropriate conditions.
The numerical experiments on some simulated convex composite optimization problems and sparse inverse covariance matrix estimation problems illustrated that our proposed algorithm is  promising.
At last, it should be mentioned that the nonergodic iteration complexity of the proposed algorithm can be easily proved, but it beyond the scope of this paper.

\section*{Acknowledgements}
The work of Y. Xiao is supported by the National Natural Science Foundation of China (Grant No. 11971149).

\section*{References}
\bibliography{mgadmm}

\end{document}